%% file: main.tex
\documentclass[10pt]{article}

\usepackage{amsmath,amsthm,amssymb}
\usepackage{mathrsfs}
\usepackage{mathtools}
\usepackage{thmtools} 
\usepackage{cancel}
\usepackage{proof}    
\usepackage{enumitem}

\usepackage[hidelinks]{hyperref}
\usepackage{graphicx, xcolor} 
\usepackage{booktabs} 
\usepackage{titlesec} 
\usepackage{sidecap}
\usepackage[font=small, labelfont=bf, labelsep=period]{caption}

\usepackage{tikz}
\tikzset{every picture/.style={line width=0.5pt}} 

\font\bigbold=cmbx12

\font\smallheader=cmssbx10 

\def\maketitle#1#2#3#4{
  \centerline {\bigbold #1}
  \medskip
  \centerline {\eightpt #2}
  \medskip
  \centerline {\tensc #3}
  \medskip
  \centerline {\sl #4}
  \bigskip
}

\titleformat{\section}{\smallheader}{\thesection.}{0.5em}{}
\titlespacing{\section}{0em}{1.5\bigskipamount}{\medskipamount}
\titleformat{\subsection}{\smallheader}{\thesubsection.}{0.5em}{}
\titlespacing{\subsection}{0em}{\bigskipamount}{\medskipamount}
\titleformat{\paragraph}[runin]{\bf}{}{0em}{}[.]
\titlespacing{\paragraph}{0em}{\medskipamount}{*0.75}

\font\tensc=cmcsc10

\font\eightpt=cmr8

\def\case#1. {{\it Case #1}.\enspace}

\def\parlabel#1. {\medskip\noindent{\bf #1.}\enspace}

\def\sqr#1#2{{\vcenter{\vbox{\hrule height.#2pt
        \hbox{\vrule width.#2pt height#1pt \kern#1pt
          \vrule width.#2pt}
        \hrule height.#2pt}}}}

\def\slug{\quad\hbox{\kern1.5pt\vrule width2.5pt height6pt depth1.5pt\kern1.5pt}\medskip}

\setlist[enumerate]{itemsep=\smallskipamount,parsep=0pt,label={\rm \roman*)}}
\setlist[itemize]{itemsep=\smallskipamount,parsep=0pt}

\newtheorem{thm}{Theorem}
\newtheorem{cor}[thm]{Corollary}

\newtheorem{lem}[thm]{Lemma}

\theoremstyle{definition}

\theoremstyle{remark}

\DeclareMathOperator*{\argmax}{arg\;max}

\DeclareMathOperator{\HS}{H}

\DeclareMathOperator{\Fr}{F}
\DeclareMathOperator{\Can}{C}
\DeclareMathOperator{\Rig}{R}

\newcommand{\II}{\mathbb{II}}
\newcommand{\III}{\mathbb{III}}

\newcommand{\ind}[1]{\1_{[#1]}}
\newcommand{\bigind}[1]{\1_{\left[#1\right]}}

\newcommand{\disteq}{\stackrel{\L}{=}}

\input{cmds.tex}
\renewcommand{\K}{\mathrm{K}}

\renewcommand{\mathbb}{\mathbf}
\renewcommand{\mathbbm}{\mathbf}
\renewcommand{\V}{\mathbf{V}}
\renewcommand{\e}{\epsilon}

\begin{document}

\maketitle{The Horton-Strahler Number of Conditioned Galton-Watson Trees}{}{Anna Brandenberger, Luc Devroye and Tommy Reddad}{School of Computer Science, McGill University}

\medskip

\[
  \vbox{
    \hsize 5.5 true in
    \noindent{\bf Abstract.}\enskip The Horton-Strahler number of a tree is a measure of its branching complexity; it is also known in the literature as the register function. We show that for critical Galton-Watson trees with finite variance conditioned to be of size $n$, the Horton-Strahler number grows as $\frac{1}{2}\log_2 n$ in probability. We further define some generalizations of this number. Among these are the \textit{rigid} Horton-Strahler number and the $k$-ary register function, for which we prove asymptotic results analogous to the standard case.
    
  }
\]

\smallskip

\section{Introduction}\label{sec:intro}

\no 
\textsc{Rooted trees, i.e., connected} acyclic graphs with one node distinguished as the root, are one of the most important structures in graph theory and computer science. Many possible functions can be defined on them, one of which is the Horton-Strahler number. It was originally conceived by geologists to classify real-world river networks and has since then been applied in multiple fields; for instance, it is known as the register function in computer science. The study of its asymptotics for various families of trees has seen considerable attention.

\paragraph{The Horton-Strahler number} For a node $u$ in a rooted tree $T$, let $T[u]$ denote the subtree of $T$ rooted at $u$. We can then recursively define the \emph{Horton-Strahler number} $\HS(T)$ of a tree $T$ as follows: 
\begin{enumerate} 
    \item if the root has no children and $|T| = 1$, then $\HS(T) = 0$,
    \item otherwise, letting $S$ be the set of children of the root, the Horton-Strahler number of the tree takes on the maximum of the Horton-Strahler numbers of the subtrees rooted at children of the root, plus one if two or more children attain the same maximal Horton-Strahler number:
    \begin{equation}\label{eq:hsdef}
        \HS(T) = \max\{\HS(T[u]): u \in S\} + \bigind{|\argmax_{u \in S} \HS(T[u]) | > 1 }.
    \end{equation}
\end{enumerate}

\paragraph{Background} The Horton-Strahler number was introduced in 1845 by Robert E.\ Horton~\cite{horton} and redefined by Arthur N.\ Strahler~\cite{strahler} in the context of hydrogeomorphology. This field represents a river network as a tree with a planar embedding, where the point furthest downstream corresponds to the root and junctions between two streams correspond to nodes in the tree. In his original work~\cite{horton}, Horton described a geometric decay of the number of branches of increasing Horton-Strahler order in a large river basin. Empirical findings from classical geological studies showed that in fact, many other key physical characteristics of river networks (e.g., basin area, stream width and length, flow velocity, etc.) can be modelled using the Horton-Strahler number~\cite{rodriguez1992power,rodriguez2001fractal}. 

\begin{SCfigure}[1.5][hbtp]  
    \centering
    \input{T_completebinary.tex}
    \caption{A visualization of a tree with Horton-Strahler number 4, where all nodes in the tree are collapsed into edges, other than the root and nodes forming an embedded complete binary tree. In fact, the Horton-Strahler number of a tree is \textit{equal to} the height of the largest embedded complete binary tree.}\label{fig:T-completebinary}
\end{SCfigure}
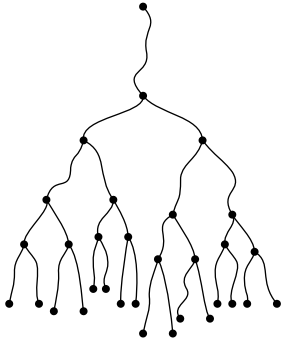

In computer science, the Horton-Strahler number is known as the register function or register number~\cite{ershov1958programming}, modulo the value at the leaf. It is equal to $\HS(T) + 1$ and corresponds to the minimum number of \textsc{cpu} registers needed to evaluate an expression tree. The probability and theoretical computer science communities have mostly devoted their attention to the register function of random equiprobable binary trees --- Catalan trees.
Already in 1966, Shreve~\cite{shreve} made some conjectures about its value based on simulations in a random topology model equivalent to a uniform distribution on planar binary trees.
Flajolet et.\ al.~\cite{flajolet}, Kemp~\cite{kemp1979average} and Meir et.\ al.~\cite{meir-moon-pounder} independently found the register function of a Catalan tree with $n$ leaves to be $\log_4 n + O(1)$. Later, Devroye and Kruszewski~\cite{krusz-paper} offered a simple probabilistic proof of this result. As for other families of trees, Flajolet and Prodinger showed similar asymptotics for Motzkin trees~\cite{flajolet1986register}.

Many quantities related to the Horton-Strahler number have been studied, mostly in the Catalan tree setting. 
Moon and others investigated the behaviour of the bifurcation ratio, i.e., ratio of number of branches with successive Horton-Strahler numbers~\cite{moon, wang1991large, yekutieli1994horton}.
The Horton-Strahler numbers have also been related to the self-similar (fractal) structure of trees. They are connected to the Horton pruning operation, which iteratively erases a tree; Burd et.\ al.~\cite{burd2000self} studied this pruning operation for critical binary Galton-Watson trees. Other references on this topic can be found in the review by Kovchegov and Zaliapin~\cite{kovchegov2020random}. 

In our work, we consider a generalization of the Horton-Strahler number to general rooted trees. Another such definition for trees with any number of children was given by Auber et.\ al.~\cite{auber2004new}. Drmota and Prodinger~\cite{drmota2006register} showed that the distribution of this number for a uniformly chosen $t$-ary tree is also highly concentrated around $\log_4 n$.

\paragraph{Galton-Watson processes} These processes were first studied in the context of disappearance of family names in 1845 by Bienaym\'e~\cite{bienayme1845} and in 1874 by Galton and Watson~\cite{galtonwatson1874}. A Galton-Watson tree \cite{athreya1972branching} with offspring distribution $\xi$ is a rooted ordered tree in which each node reproduces according to $\xi$, i.e., has $i$ children with probability $p_i = \P\{\xi = i\}$. Excluding the distribution where $p_1 \equiv 1$, it is well known that these trees are finite with probability one if and only if $\E\{\xi\} \leq 1$. Simultaneously, the first moment of the size of trees with $\E\{\xi\} = 1$ is infinite. We will consider these critical trees with mean $\mu \ceq \E\{\xi\} = 1$ and variance $\sigma^2 \ceq \V\{\xi\} \in (0, \infty)$.

Let $T$ denote a $\xi$-Galton-Watson tree, which from now on we will call \emph{unconditional} Galton-Watson tree. We distinguish this type of tree from from the trees we study in this paper, that are conditioned to have size $|T| = n$. We will denote such a \emph{conditional} tree as $T_n$. Conditional Galton-Watson trees~\cite{kennedy1975} are an especially interesting structure to study, as certain offspring distributions have been shown to correspond to families of ``simply-generated trees"~\cite{meirmoon1978}, such as  $k$-ary trees, Motzkin trees and planted plane trees. Picking a tree uniformly at random from such a family is thus equivalent to generating a corresponding conditional Galton-Watson tree.

\smallskip 

In this paper, we will show that for a critical conditional Galton-Watson tree $T_n$ with variance $\sigma^2 \in (0, \infty)$,
the Horton-Strahler number of the root satisfies 
\[\frac{\HS(T_n)}{\log_2 n} \to \frac{1}{2}\]
in probability as $n \to \infty$. This expression synthesizes all previously known first order results; however, higher order concentration information is not presented here. 

Furthermore, we present other definitions of possible Horton-Strahler numbers (see section~\ref{sec:otherdef}), and offer partial or full results about these numbers. For instance, included in these definitions is a $k$-ary register function, which corresponds to a computational model in which each register in a computer takes $k$ inputs to produce an output in one step. We show that the $k$-ary register function of a critical conditional Galton-Watson tree grows as
\[ \frac{\log_2 \log_2 n}{\log_2 k/2}\]
in probability as $n \to \infty$.

\section{Unconditional Galton-Watson Trees}
\label{sec:unconditional}
\no 
We begin by determining the distribution of the Horton-Strahler number of a Galton-Watson tree with no size conditioning. These results, particularly Theorem~\ref{thm:uncond}, will be crucial to later proofs of the upper and lower bounds in Sections~\ref{sec:lowerbound} and \ref{sec:upperbound}. Indeed, unconditional Galton-Watson trees are a part of the construction of Kesten's limit tree~\cite{kesten1986}, which we will introduce and heavily use in the next section. 

\smallskip 
Let us first define some notation for an unconditional Galton-Watson tree with mean $\mu = 1$ and variance $0 < \sigma^2 < \infty$. Let the generating function of the offspring distribution $\xi$ be $f(s) = \sum_{i=0}^\infty p_i s^i$  on $0 \leq s \leq 1$, where we recall that $p_i = \P\{\xi = i\}$. Furthermore, let us define for $i \in \N$, the probability that the Horton-Strahler number of the root is $i$,
\begin{equation}
    q_i = \P\{\HS(T) = i\},
\end{equation}
as well as the partial sums 
\begin{equation*}
    q_i^+ = \sum_{k=i}^\infty q_k \qquad \ad \qquad q_i^- = \sum_{k=0}^i q_k,
\end{equation*}
which are involved in the recursion for $q_i$ and other elements of proofs in this section. 

Our first lemma formalizes the intuition that nodes with one child are irrelevant to the Horton-Strahler number, as these nodes simply pass on the number of their only child. We will show that, in fact, altering the offspring distribution by removing the probability of having one child still preserves the original Horton-Strahler number. 

\begin{lem}\label{lem:removep1} 
Let $\xi$ be an offspring distribution with $\mu=1$ and $0 < \sigma^2 < \infty$. Let $\zeta$ be an altered distribution defined by $\P\{\zeta = 1\} = 0$ and for $i \neq 1$, $\P\{\zeta = i\} = p_i/(1-p_1)$. Then, defining $\HS$ and $\HS'$ to be respectively the Horton-Strahler number of an unconditional $\xi$- and $\zeta$-Galton-Watson tree, we have 
\begin{equation}
    \HS' \disteq \HS.
\end{equation}
\end{lem}
\no This can be proven via induction on $i \in \N$. The details of the proof are relatively tedious and therefore relegated to Appendix~\ref{appendix:uncond}. 
Armed with this lemma, we will be able to simplify proofs by trivially removing single-child nodes from any offspring distribution we are given without changing the distribution of the Horton-Strahler number; this altered distribution still has $\mu=1$ and $0 < \sigma^2 < \infty$.
\smallskip 

We now come to the main theorem regarding the Horton-Strahler number of unconditional Galton-Watson trees: this number has an exponentially decreasing probability. This has already been shown for Catalan trees; for instance, see Devroye and Kruszewski~\cite{krusz-paper}. We note that a random Catalan tree can be generated as a Galton-Watson tree with offspring distribution $p_0 = 1/4$, $p_1 = 1/2$, $p_2 = 1/4$. 

\paragraph{A Simple Proof for Catalan Trees} We can prove very simply by induction that for a Catalan tree $T$, 
\[\P\{\HS(T) = x\} = 2^{-(x+1)}.\]
\begin{proof} First, we have from Lemma~\ref{lem:removep1} that the Horton-Strahler number of a Catalan tree is distributed identically to that of a fully binary tree generated via the distribution $p_0 = 1/2$, $p_1 = 0$, $p_2 = 1/2$. 
The base case is trivial: 
\[\P\{\HS(T) = 0\} = p_0 = 1/2.\] 
Then, supposing that $\P\{\HS(T) = x-1 \} = 2^{-x}$, we have 
\begin{align*}
    p \ceq \P\{\HS(T) = x\} &= p_2 \Big( (2^{-x})^2 + 2 p \sum_{i=0}^{x-1} 2^{-(i+1)} \Big) \\ 
    &= \frac{1}{2} \bpar{ (2^{-x})^2 + 2 p (1 - 2^{-x}) },
\end{align*}
which can be simplified to $p = \frac{1}{2}(2^{-x})$, completing the proof.
\end{proof}

Curiously, trees generated from all other critical offspring distributions give rise to a very similar formulas for their Horton-Strahler numbers!  

\begin{thm}\label{thm:uncond}
Let $T$ be an unconditional Galton-Watson tree with offspring distribution $\xi$ with $\mu = 1$ and $0 < \sigma^2 < \infty$. Then, with $x \in \N$,
\begin{equation}\label{eq:unconditional}
    \P\{\HS(T) = x\} = \Theta(2^{-x + o(x)})
\end{equation}
as $x \to \infty$.
\end{thm}
\no The main ingredient of the proof of this theorem is the recursion \eqref{eq:probrecursion} obtained in the proof of Lemma~\ref{lem:removep1}. Without loss of generality, assuming that $p_1 = 0$, the probability that the Horton-Strahler number of the root is $i \in \N$ satisfies
\[q_i = \sum_{\ell=2} p_\ell \bigg(\ell q_i (q_{i-1}^-)^{\ell - 1} + \sum_{r = 2}^\ell \binom{\ell}{r} q_{i-1}^r (q_{i-2}^-)^{\ell - r} \bigg).\]
This yields an inequality for $q_i$ involving $q_i^+$ and $q_{i-1}$, and the result then follows from some computations. The details are relatively technical and provide little intuition; we thus include the proof in Appendix~\ref{appendix:uncond}. 
\smallskip

We can also show that the Horton-Strahler number of any critical unconditional Galton-Watson tree (including those with infinite variance) have an exponentially decreasing upper bound.
\begin{thm}\label{thm:unconditionalmonotone}
For all critical unconditional Galton-Watson trees with any $\sigma^2 \in [0,\infty]$, we have that $q_i$ is monotonically decreasing as $i \to \infty$. Also, for $i \geq 2$,
\[q_i \leq \frac{p_0}{2^{i/2}}.\]
\end{thm}
\begin{proof}
Since $q_i \leq q_{i-1}^2 / 2q_i^+,$ we have that 
\[q_i^2 \leq q_i q_i^+ \leq q_{i-1}^2/2.\]
Thus, $q_i$ is monotone and $q_i \leq q_0 / 2^{i/2} = p_0 / 2^{i/2}$.
\end{proof}

\section{Lower Bound via Kesten's Limit Tree}
\label{sec:lowerbound}
\no 
In order to prove the lower bound, we will be using the notion of Kesten's limit tree~\cite{kesten1986}. This limit tree $T^\infty$ is an infinite tree consisting of a \emph{central spine} and unconditional trees hanging off the spine. To define how this tree and its spine is generated, we define a new size-biased random variable $\zeta$ as $\P\{\zeta = i\} = ip_i$, where $p_i$ correspond to our original offspring distribution $\xi$. This is a valid probability distribution since we are considering distributions $\xi$ with mean $\E\{\xi\} = \sum_{i =1}^\infty ip_i = 1$. The spine of Kesten's tree thus consists of one node on each level that reproduces according to $\zeta$; note that $\zeta \geq 1$, making this tree infinite. One of the children of each spine node, picked uniformly at random, is assigned to be the spine node of the next level, and all others are roots of an unconditional Galton-Watson tree with offspring distribution $\xi$. There is an important way in which conditional Galton-Watson trees converge to this infinite tree. Denote for any tree $T$ the finite tree $\tau(T, h)$, which is $T$ cut off after level $h$. We have that for all fixed heights $h>0$ and all trees $t$, 
\begin{equation}\label{eq:kesten}
    \lim_{n\to\infty} \P\{\tau(T_n, h) = \tau(t,h)\} = \P\{\tau(T^\infty, h) = \tau(t,h)\}.
\end{equation}
In the case where the variance of $\xi$ is finite, Benedikt Stufler proved that this convergence does not in fact require the truncation height $h$ to be a constant --- it can also depend on the size $n$ of the tree. Theorem~5.2 of Stufler's paper \cite{stufler} states that the sequence of truncation heights $h_n$ must then satisfy \[h_n = o(\sqrt{n}).\]

\paragraph{Intuitive ``proof"} The view of the conditional Galton-Watson tree converging to a Kesten tree gives us the intuition for the Horton-Strahler number of the root being $\frac{1}{2}\log_2 n$. It is well known~\cite{flajolet1982average} that conditional Galton-Watson trees have expected height $O(\sqrt{n})$. Then, for approximation, let's consider a Kesten tree cut off at height $\sqrt{n}/\sigma$, denoted $\tau(T^\infty, \sqrt{n}/\sigma)$. This tree has a spine of length $\sqrt{n}/\sigma$ and each spine node indexed $i=1, \dots, \sqrt{n}/\sigma$ has $\zeta_i - 1$ unconditional Galton-Watson trees hanging from it. We can define the $j$-th unconditional tree hanging from spine node $i$ as $T_{ij}$, $j = 1, \dots, \zeta_i - 1$. The Horton-Strahler number of the root then satisfies
\[\max_{ij} \HS(T_{ij}) \leq \HS\bpar{\tau(T^\infty, \sqrt{n}/\sigma) } \leq \max_{ij} \HS(T_{ij}) + 1. \]
We therefore have
\begin{equation*}
\begin{aligned}
    \P\bcurly{\HS\bpar{\tau(T^\infty, \sqrt{n}/\sigma) } \geq x} 
    &\leq \P\Big\{\max_{ij}  \HS(T_{ij}) + 1 \geq x \Big\} \\ 
    &\leq \E\Big\{\sum_{i=1}^{\sqrt{n} / \sigma} \sum_{j=1}^{\zeta_i - 1} \bigind{\HS(T_{ij}) \geq x - 1} \Big\}.
\end{aligned}
\end{equation*}
Using Wald's inequality~\cite{wald1944cumulative} with $\E\{\zeta_i\} = \sigma^2 + 1$, and noting that $T_{ij}$ are all i.i.d.\ and distributed as $T$,
\begin{align}
    \P\bcurly{\HS\bpar{\tau(T^\infty, \sqrt{n}/\sigma) } \geq x} &\leq \frac{\sqrt{n}}{\sigma} \sigma^2 \P\{\HS(T) \geq x - 1\} \nonumber \\ 
    &= \sigma \sqrt{n} 2^{-x + 2 + o(x)}, \label{eq:intuitiveub}
\end{align}
which tends to zero if $x = (1/2 + \e) \log_2 n$ for some $\e > 0$.

For the lower bound, the following is slightly incorrect, as it assumes that each spine node has at least one hanging tree. We present it here to illustrate the main idea; see the proof of Theorem~\ref{thm:lowerbound} for the rigorous statement.
\begin{equation*}
\begin{aligned}
    \P\bcurly{\HS\bpar{\tau(T^\infty, \sqrt{n}/\sigma) } \leq x} &\leq \P\Big\{\max_{ij}  \HS(T_{ij}) \leq x \Big\} \\ 
    &\leq \P\Big\{ \bigcap_{i=0}^{\sqrt{n}/\sigma} [\HS(T_{i1} \leq x ]\Big\} \\ 
    &\leq (1 - \P\{H(T) > x\})^{\sqrt{n}/\sigma} 
\end{aligned}
\end{equation*}
since the unconditional trees $T_{ij}$ are i.i.d.\ distributed as $T$. Then, applying Theorem~\ref{thm:uncond} yields 
\begin{align}
    \P\bcurly{\HS\bpar{\tau(T^\infty, \sqrt{n}/\sigma) } \leq x} 
    &\leq (1 - 2^{-x + o(x)})^{\sqrt{n}/\sigma} \nonumber \\
    &\leq \exp(-\sqrt{n}/\sigma 2^{-x + o(x)}), \label{eq:intuitivelb}
\end{align}
which tends to zero if $x = (1/2 - \e) \log_2 n$ for $\e > 0$.

\medskip   

We thus have that the Horton-Strahler number of Kesten's limit tree truncated at level $\sqrt{n}/\sigma$ tends to $\frac{1}{2}\log_2 n$. Intuitively, since conditional Galton-Watson trees converge to this limit tree as $n \to \infty$, in the sense of \eqref{eq:kesten}, the Horton-Strahler number of our conditional trees should be the same as $n \to \infty$. 
Indeed, the lower bound for conditional Galton-Watson trees can be proven using the same method as what we have just used in this intuitive proof. 

\begin{thm}\label{thm:lowerbound}
Given a critical conditional Galton-Watson tree with offspring distribution $\xi$ with variance satisfying $0 < \sigma^2 < \infty$, then 
\begin{equation}\label{eq:lowerbound}
    \P\{\HS(T) \leq x \mid |T| = n\} \to 0
\end{equation}
if $x = (1/2 - \e) \log_2 n$ for any $\e > 0$.
\end{thm}
\begin{proof}
Recall some notation: $T$ denotes an unconditional Galton-Watson tree, and $T^\infty$ denotes Kesten's limit tree. For some integer $\ell$, we can cut off $T^\infty$ by taking all the nodes on the spine including the node at distance $\ell$ from the root, but no further. To this, we can add all unconditional trees hanging from these $\ell + 1$ spine nodes. This forms a finite tree that we denote $T^\infty_\ell$; a diagram is shown in Figure~\ref{fig:T-inf-ell}.

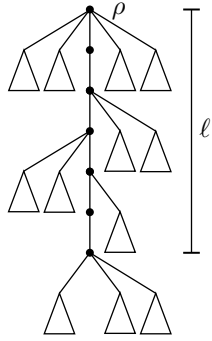
\begin{SCfigure}[1.5][hbtp]  
    \centering
    \input{Tinf-cutoff.tex}
    \caption{A visualization of $T^\infty_\ell$: Kesten's limit tree $T^\infty$ rooted at a node $\rho$, with its spine truncated after the spine node on level $\ell$. Each triangle represents a hanging unconditional Galton-Watson tree.}\label{fig:T-inf-ell}
\end{SCfigure}

Let $h(T)$ denote the height of a tree $T$, and let $T_n$ denote the tree $T$ conditioned to have size $|T| = n$. For some $x \geq 1$, define the three probabilities
\begin{equation}\label{eq:lowerboundI-II-III}
\begin{aligned}
    \I &\ceq \P\bcurly{ h \bpar{T^\infty_{\sqrt{n}/\log n}} > \sqrt{\frac{n}{\log n}} }, \\ 
    \II &\ceq \P\bcurly{ \tau\bpar{T_n, \sqrt{\frac{n}{\log n}}} \neq \tau\bpar{T^\infty,  \sqrt{\frac{n}{\log n}}} } , \\ 
    \III &\ceq \P\bcurly{ h \bpar{T^\infty_{\sqrt{n}/\log n}} \leq \sqrt{\frac{n}{\log n}}, \tau\bpar{T_n, \sqrt{\frac{n}{\log n}}} = \tau\bpar{T^\infty,  \sqrt{\frac{n}{\log n}}} , \HS(T_n) \leq x }.
\end{aligned}
\end{equation}
We have
\begin{equation}\label{eq:lowerbounddecomposition}
    \P\{\HS(T_n) \leq x\} \leq \I + \II + \III.
\end{equation}
Let's start with the two terms that do not depend on $x$. As discussed earlier in this section, Stufler~\cite{stufler} showed that  $\P\{\tau(T_n, h_n) \neq T^\infty, h_n) \} = o(1)$ if $h_n = o(\sqrt{n})$. Thus, since $\sqrt{n/\log n} = o(\sqrt{n})$, we have 
\begin{equation}\label{eq:lbII}
    \II = o(1).  
\end{equation}
Now for the first term, let $\ell = \sqrt{n}/\log n$. We recall our notation of $\zeta_i$ as the number of children of the $i$-th node on the spine, and further define $T_{ij}$ to be the $j$-th Galton-Watson tree hanging from this $i$-th spine node. These unconditional trees are i.i.d.\ and distributed as $T$. We thus have 
\begin{equation*}
        \I \leq \E \bigg\{ \sum_{i=0}^\ell \sum_{j=1}^{\zeta_i - 1} \bigind{h(T_{ij}) \geq \sqrt{n/\log n} - \ell } \bigg\},
\end{equation*}
and by Wald's identity~\cite{wald1944cumulative}, as $\E\{\zeta\} = \sigma^2 + 1$,
\begin{align*}
    \I \leq \bpar{\frac{\sqrt{n}}{\log n} + 1} \sigma^2  \P\bcurly{\HS(T) \geq \sqrt{\frac{n}{\log n}} - \frac{\sqrt{n}}{\log n} }.
\end{align*}
Finally, using Kolmogorov's estimate~\cite{kestenney1966galton,lyons2017probability}, this grows as
\begin{equation}\label{eq:lbI}
    \I \sim \frac{\sqrt{n}}{\log n} \sigma^2 \frac{2}{\sigma^2 \sqrt{n/\log n}} = \frac{2}{\sqrt{\log n}},
\end{equation}
which approaches zero as $n \to \infty$.

For the third term, note that the first and second events included in the probability imply that the truncated Kesten limit tree $T^\infty_{\ell}$ at $\ell = \sqrt{n}/\log n$ is completely included in our conditional Galton-Watson tree $T_n$. This inclusion implies that $\HS(T^\infty_{\ell}) \leq \HS(T_n)$, which yields
\begin{equation*}
    \III \leq \P\bcurly{\HS(T^\infty_{\ell}) \leq x }.
\end{equation*}
Note that now this is exactly the form of what we had in the intuitive proof! We can thus follow exactly in the steps outlined in the derivation of \eqref{eq:intuitivelb}. Let $T_{ij}$ again denote the $j$-th unconditional Galton-Watson tree hanging from the $i$-th spine node. Let $N = \sum_{i=0}^\ell (\zeta_i - 1)$ be the number of hanging trees, which has mean $\E\{N\} = (\ell + 1)\sigma^2$. Note that the hanging trees are i.i.d.\ distributed as $T$, the number $N$ is a sum of $\ell + 1$ i.i.d.\ random variables. Therefore, using the law of large numbers, we can bound 
\begin{equation}\label{eq:numtrees}
\begin{aligned}
    \P\bcurly{N < (\ell+1) \frac{\sigma^2}{2} } 
    &\leq \P\bcurly{ \left\lvert N - (\ell + 1) \sigma^2 \right\rvert >  (\ell + 1) \frac{\sigma^2}{2} } \\ 
    &= o(1).
\end{aligned}
\end{equation}
We then have 
\begin{align}
    \III &\leq \P\Big\{\max_{ij}  \HS(T_{ij}) \leq x, N \geq (\ell + 1)\frac{\sigma^2}{2} \Big\} + \P\Big\{\max_{ij}  \HS(T_{ij}) \leq x, N < (\ell + 1)\frac{\sigma^2}{2} \Big\} \\ 
    &\leq (1 - \P\{H(T) > x\})^{\ell \sigma^2/2} + o(1) \nonumber \\ 
    &\leq \exp\bpar{-\frac{\sigma^2}{2}\frac{\sqrt{n}}{\log n} 2^{-x + o(x)}} + o(1), \label{eq:lbproof}
\end{align}
which tends to zero as $n \to \infty$ for $x = (1/2 - \e) \log_2 n$; $\e \in (0, 1/2)$.
\end{proof}

Thus, modulo some details regarding the convergence of the conditional Galton-Watson tree to Kesten's limit tree, the intuitive proof idea miraculously works to show the lower bound of our result. 
However, the upper bound cannot be shown following this proof sketch; the contributions of terms underneath any given cutoff cannot be ignored.
We will instead offer a proof based on the construction of rotationally invariant events.

\section{Upper Bound via a Rotationally Invariant Event}
\label{sec:upperbound}
\no 
For the upper bound, we  note that in order for a tree to have a Horton-Strahler number equal to $k$, we must be able to embed a complete binary tree of height $k$ in the original tree (see Figure~\ref{fig:T-completebinary}). 

We therefore immediately have a deterministic upper bound of 
\[H_n \leq \log_2\Big(\frac{n+1}{2}\Big)\]
for any tree of size $n$. We seek to do better than this.

\smallskip 

\paragraph{Random walk view of a Galton-Watson tree} Numbering the nodes in a Galton-Watson tree $T$ in preorder traversal, each node has a tree degree $\xi_i$ independently distributed as $\xi$. This sequence of random variables defines a tree of size
\begin{equation}\label{eq:treesizedef}
    \begin{aligned}
    |T| &= \min \{t > 0: 1 + (\xi_1 - 1) + \cdots + (\xi_t - 1) = 0\}\\ 
        &= \min \Big\{t > 0: \sum_{i=1}^t \xi_i = t - 1\Big\}.
    \end{aligned}
\end{equation}
Thus, for a tree of size $|T| = n$, the event 
\begin{equation}\label{eq:Adef}
    A = \Big\{ \sum_{i=1}^n \xi_i = n-1 \Big\}
\end{equation}
must be true, and furthermore, the random walk must stay positive until the last time step where it reaches $-1$, i.e., for all $t<n$, $\sum_{i=1}^t (\xi_i - 1) \geq 0$.

\paragraph{Rotationally invariant events} Any event $B$ on a tree $T$ of size $|T| = n$ is determined by the degree sequence $\xi_1, \dots, \xi_n$ of the tree. We say that this event $B$ satisfies rotation invariance if it remains true when applied to $\xi_i, \dots, \xi_n, \xi_{1}, \dots, \xi_{i-1}$ for all $i$. We have a powerful tool to deal with such events on a conditional Galton-Watson tree $T$. Letting $A$ be the event defined in \eqref{eq:Adef} and using Dwass' cycle lemma \cite{dwass1969}, it can easily be shown~\cite{OURPAPER} that 
\begin{equation}\label{eq:rotation}
    \P\{B \mid |T| = n\} = \P\{B \mid A\}.
\end{equation}
Note that a rotation of these random variables $\xi_i, \dots, \xi_n, \xi_1, \dots, \xi_{i-1}$ defines a forest in which the last tree is unfinished --- let us call it the $i$-forest. Each of the trees in the forest is obtained as follows: for a tree starting at index $i$, we simply pick the first index $j > i$ for which the degree sequence $\xi_i, \dots, \xi_j$ defines a tree, i.e., satisfies \eqref{eq:treesizedef} for an appropriate tree size. Denote this tree size by $|T(\xi_i, \xi_{i+1}, \dots)|$. If there is no such index $j \in \{1, \cdots, n\}$, then the tree starting at $i$ is undefined. 

In wielding \eqref{eq:rotation}, we are aided by the fact that we have an exact asymptotic limit for $\P\{A\}$ due to Kolchin \cite{kolchin1986}. Letting the period of $\xi_1$ be 
\[h = \gcd\{i \geq 1 : p_i > 0\},\]
we have
\begin{equation}\label{eq:kolchinA}
    \P\{A\} \sim \frac{h}{\sigma \sqrt{2\pi n}}.
\end{equation}

In order to make use of \eqref{eq:rotation} and \eqref{eq:kolchinA} in our current setting, we must define a rotationally invariant event that is related to the Horton-Strahler number. 
Given some i.i.d.\ degree sequence $\xi_1, \dots, \xi_n$, for each $i \in \{ 1, \dots, n\}$, let $T_i$ be the first tree in the $i$-forest. Define $\eta_i$ to be the Horton-Strahler number of this tree: 
\[\eta_i = \HS(T_i). \] 
If $T_i$ is unfinished, let $\eta_i = 0$. Then, define
\begin{equation}
    \HS^*(T_n) = \max_{1 \leq i \leq n} \eta_i,
\end{equation}
a rotationally invariant quantity. Since $\HS(T_n) = \eta_1$ given $|T| = n$, this is an upper bound to the Horton-Strahler number: \[\HS(T_n)\leq \HS^*(T_n).\] 

The upper bound we seek to show will follow from the following theorem linking the Horton-Strahler number of a conditional Galton-Watson tree to the $\eta_1$ we just defined. 

\begin{thm}\label{thm:upperbound}
Given a critical conditional Galton-Watson tree with offspring distribution $\xi$ and $0 < \sigma^2 < \infty$, for some constant $c$, 
\begin{equation}\label{eq:upperboundeq}
\P\{\HS(T) \geq x \mid |T| = n\} \leq c \sqrt{n} \P\{\eta_1 \geq x - 2\}.
\end{equation}
\end{thm}
\begin{proof}
We have from \eqref{eq:rotation} that 
\begin{equation}\label{eq:upperboundstart}
\begin{aligned}
    \P\{\HS(T) \geq x \mid |T| = n \} 
    &\leq \P\{\HS^*(T) \geq x \mid |T| = n\} \\ 
    &= \frac{\P\{\max_i \eta_i \geq x, A\}}{\P\{A\}}.
\end{aligned}
\end{equation}
If $\eta_i \geq x$, then there must exist some $j$ with $\eta_j \geq x-2$ such that the tree size determined by $\xi_j, \xi_{j+1}, \dots$ is at most $n/4$. This is since there must be at least four disjoint subtrees with Horton-Strahler number greater or equal to $x-2$, and the smallest of these subtrees must have size $|T(\xi_j, \xi_{j+1}, \dots)|$ at most $n/4$. We thus define
\begin{equation}\label{eq:in_over4_def}
    \eta_i^{(n/4)} = \begin{cases} \eta_i & \text{if } |T(\xi_i, \xi_{i+1}, \dots)| \leq n/4, \\ 
    0 & \text{otherwise}\end{cases}.
\end{equation}
for all $i$. Thus, the numerator of \eqref{eq:upperboundstart} satisfies
\begin{align}
    \P\{\max_i \eta_i \geq x, A\} &\leq \P\left\{\max_i \eta_i^{(n/4)} \geq x-2, A \right\} \nonumber \\ 
    &\leq 4 \P\left\{\max_{1\leq i \leq n/4} \eta_i^{(n/4)} \geq x-2, A \right\}. \label{eq:max_etai_tbc}
\end{align}
Next, we define the cumulative sums $S_0 = 0$, and for $1\leq i \leq n$,
\begin{equation}
    S_i = \xi_1 + \cdots + \xi_i.
\end{equation}
We define $\eta_i^*$ for $i \leq n/4$ as follows:  
\begin{enumerate}
    \item if the first tree in the $i$-forest is defined within the nodes $\xi_i, \dots, \xi_{n/2-1}$, let $\eta_i^* = \eta_i$;
    \item if the first tree in the $i$-forest is unfinished, let $\eta_i^*$ be the maximal Horton-Strahler number for any subtree of node $i$, i.e., for any tree occurring in the forest defined by $\xi_{i+1}, \dots, \xi_{n/2-1}$.
\end{enumerate}
Note that we again consider the Horton-Strahler number of any unfinished tree in this forest to be zero. As such, the tree defined by $\eta_i^*$ has size less than $n/2$, and for all $1 \leq i \leq n/4$,
\[\eta_i^{(n/4)} \leq \eta_i^* \leq \eta_i.\]
Then, defining the events 
\[A_i = \left\{ \max_{1\leq j <i} \eta_j^* < x-2, \eta_i^* \geq x-2 \right\} \]
and 
\[D_i = \left\{\max_{1\leq j <i} \eta_j^* < x-2, \eta_i^* \geq x-2, A \right\} = A_i \cap A, \]
the inequality \eqref{eq:max_etai_tbc} becomes 
\begin{equation*}
    \P\{\max_i \eta_i \geq x, A\} 
    \leq 4 \P\left\{ \max_{1 \leq i \leq n/4} \eta_i^* \geq x-2, A \right\}  \leq 4 \sum_{i=1}^{n/4} \P\left\{ D_i \right\}.
\end{equation*}
We must now analyze the event $D_i$: 
\begin{align*}
    \P\{D_i\} &= \P\Big\{A_i, \sum_{j=1}^n \xi_j = n-1 \Big\} \\ 
    &= \sum_{k=-\infty}^\infty \P\Big\{ A_i, S_{n/2 - 1} = k, \sum_{j=n/2}^n \xi_j = n-k-1 \Big\}\\
    &= \sum_{k=-\infty}^\infty \P\Big\{ A_i, S_{n/2 - 1} = k\Big\} \cdot \P\Big\{ \sum_{j=n/2}^n \xi_j = n-k-1 \Big\} \\
    &\leq  \sum_{k=-\infty}^\infty \P\Big\{ A_i, S_{n/2 - 1} = k\Big\} \cdot \sup_{k} \P\Big\{ \sum_{j=n/2}^n \xi_j = k \Big\},
\end{align*}
where the third equality holds by independence of the $\xi_i$'s. In order to bound this, we make use of Rogozin's inequality \cite{rogozin1961}, which we recall states that if $X_1, \dots, X_n$ are i.i.d.\ random variables and 
\[p = \sup_x \P\{X_i = x\}, \]
then 
\begin{equation}\label{eq:rogozin}
    \sup_x \P\{X_1 + \cdots + X_n = x\} \leq \frac{\alpha}{\sqrt{n(1-p)}}
\end{equation}
for some universal constant $\alpha$. In our case, we consider offspring distributions $\xi$ satisfying $0 < \sigma^2 < \infty$, which guarantees $p_0 > 0$. We therefore have $p<1$, and arrive at
\begin{equation}
    \P\{D_i\} \leq \frac{c'}{\sqrt{n}}\P\{A_i\}
\end{equation}
for some constant $c'$. 
Further defining the event 
\[B_i = \left\{ \max_{1 \leq j < i} \eta_j^* < x-2 \right\}, \]
we can write the event $A_i$ as $A_i = B_i \cap \{\eta_i^* \geq x-2\}$. The event $A_i$ can only occur either if $i=1$ or if $S_{i-1} < \min_{0 \leq \ell < i-1} S_\ell$.
If $i = 1$, we directly have 
\[ \P\{A_1\} \leq \P\{\eta_1 \geq x - 2\}.\]
If $i > 1$, 
\begin{equation*}
\begin{aligned}
    \P\{A_i\} &= \sum_{k=-\infty}^\infty \P \Big\{ B_i, S_{i-1} = k < \min_{\ell < i-1} S_\ell, \eta_i^* \geq x-2 \Big\} \\ 
    &= \sum_{k=-\infty}^\infty \P \Big\{ B_i, S_{i-1} = k < \min_{\ell < i-1} S_\ell \Big\}  \cdot \P\left\{ \eta_i^* \geq x-2 \right\} \\ 
    &\leq \sum_{k=-\infty}^{-1} \P\{S_{i-1} = k < \min_{\ell < i-1} S_\ell\} \cdot \P\left\{ \eta_i \geq x-2 \right\} \\ 
    &= \sum_{k=-\infty}^{-1} \frac{|k|}{i-1} \P\{S_{i-1} = k\} \cdot \P\left\{ \eta_1 \geq x-2 \right\} \\ 
    &= \frac{\E\bcurly{ |S_{i-1}| \ind{S_{i-1} \leq 1}}}{i-1} \P\left\{ \eta_1 \geq x-2 \right\} ,
\end{aligned}
\end{equation*}
where the last line follows from a rotational argument on $\xi_1, \dots, \xi_{i-1}$. Then, by Cauchy-Schwartz,
\begin{align*}
    P\{A_i\} &\leq \E\left\{\frac{|S_{i-1}|}{i-1} \right\} \P\left\{ \eta_1 \geq x-2 \right\} \\ 
    &\leq \frac{\sqrt{ \E\{S_{i-1}^2\} }}{i-1} \P\left\{ \eta_1 \geq x-2 \right\} \\ 
    &= \frac{\sigma}{\sqrt{i-1}} \P\left\{ \eta_1 \geq x-2 \right\}.
\end{align*}
Thus, considering the two cases $i=1$ and $i>1$, 
\begin{equation}
    D_i \leq \begin{cases} 
    \frac{c'}{\sqrt{n}} \P\left\{ \eta_1 \geq x-2 \right\} & i=1, \\ 
    \frac{c'}{\sqrt{n}} \frac{\sigma}{\sqrt{i-1}}\P\left\{ \eta_1 \geq x-2 \right\} & i>1.
    \end{cases} 
\end{equation}
Therefore, returning to the numerator of \eqref{eq:upperboundstart}, we have
\begin{equation*}
\begin{aligned}
    \P\{\max_i \eta_i \geq x, A\} &\leq 4 \P\{\eta_1 \geq x-2\} \cdot \frac{c'}{\sqrt{n}}\Big( 1 + \sum_{i=2}^{n/4} \frac{\sigma}{\sqrt{i-1}}\Big) \\ 
    &\leq c'' \P\{\eta_1 \geq x-2\}
\end{aligned}
\end{equation*}
for some constant $c''$. Finally, we have from \eqref{eq:kolchinA} that there exists another constant $c$ such that 
\begin{align*}
    \P\{\HS(T) \geq x \mid |T| = n\} &\leq \frac{c'' \P\{\eta_1 \geq x-2\}}{\P\{A\}} \\ 
    &\leq c \sqrt{n}\P\{\eta_1 \geq x-2\},
\end{align*}
completing the proof.
\end{proof}

Everything we require follows from this theorem. The following corollary gives us the upper bound of the classical Horton-Strahler number. 

\begin{cor}\label{cor:bigresult}
For a critical conditional Galton-Watson tree $T$ with $0 < \sigma^2 < \infty$, 
\begin{equation}
    \P\{\HS(T) \geq x \mid |T|=n\} \to 0
\end{equation}
if $x = (1/2 + \e)\log_2 n$ for any $\e > 0$.
\end{cor}
\begin{proof}We have that 
\[\P\{\eta_1 \geq x - 2\} \leq \P\{\HS(T(\xi_1, \xi_2, \cdots)) \geq x - 2\} ,\]
where $\HS(T(\xi_1, \xi_2, \cdots))$ is the Horton-Strahler number of the first tree in the infinite sequence $\xi_1, \xi_2, \dots$, i.e., the Horton-Strahler number of an unconditional Galton-Watson tree. Note that we have the inequality since there is a possibility for the first tree to be unfinished, in which case $\eta_1 = 0$.

Recall that we had from Theorem~\ref{thm:uncond} that
\[\P\{\HS(T)\geq x-2\} \leq 2^{-x + o(x)}\]
as $x \to \infty$ for an unconditional Galton-Watson tree $T$. Thus, by Theorem~\ref{thm:upperbound}, there exists a constant $c$ such that
\[\P\{\HS(T) \geq x \mid |T|=n \} \leq c \sqrt{n} 2^{-x + o(x)}, \]
which tends to zero if $x = (1/2 + \e)\log_2 n$, for any $\e > 0$.
\end{proof}

\section{Generalizations of the Horton-Strahler Number}
\label{sec:otherdef}
\no 
Our definition \eqref{eq:hsdef} is not the only possible one. In this definition, the number increments at each river branching where two rivers attain the same maximal flow. We can define various generalizations of this number for non-binary trees, ranging from less to more strict. We will discuss three additional natural definitions. All of them will be recursively defined from the values of all subtrees, and leaf nodes $u$ with subtree size $|T[u]| = 1$ will always have the value 0.
\begin{enumerate} 
    \item The \textit{French Horton-Strahler number}, where French refers to its source, Auber et. al.~\cite{auber2004new}. If the root of the tree $T$ has $k$ children with subtrees taking values $\Fr_1 \geq \Fr_2 \geq \cdots \geq \Fr_k \geq 0$ (sorted in decreasing order), then the tree has French Horton-Strahler number
    \begin{equation}\label{eq:french}
        \Fr(T) \ceq \max_{1 \leq i \leq k} (F_i + (i-1)).
    \end{equation}
    
    \item The \textit{Canadian Horton-Strahler number}. If the root of the tree $T$ has $k$ children with subtrees taking values $\Can_1 \geq \Can_2 \geq \cdots \geq \Can_k \geq 0$ (sorted in decreasing order), and we have $r$ children with the maximal value $\Can_1 = \cdots = \Can_r > \Can_{r + 1} \geq \cdots$, then the root has Canadian Horton-Strahler number 
    \begin{equation}\label{eq:canadian}
        \Can(T) \ceq \Can_1 + (r - 1) = \max_{1 \leq i \leq k} \Can_i + (r - 1).
    \end{equation}
    
    \item The \textit{(standard) Horton-Strahler number} studied earlier in this paper was given in \eqref{eq:hsdef}. Following similar notation as given in this list, for $k$ children with subtrees taking values $\HS_1 \geq \HS_2 \geq \cdots \geq \HS_k \geq 0$, then the Horton-Strahler number of the root is 
    \begin{equation*}
        \HS(T) \ceq 
        \begin{cases} 
        \HS_1 = \max_{1 \leq i \leq k} \HS_i  & \text{ if } k = 1, \\
        \HS_1 + \ind{\HS_1 = \HS_2}  & \text{ if } k > 1.
        \end{cases}
    \end{equation*}
    
    \item The \textit{rigid Horton-Strahler number}. Again, with the same notation of $k$ children with subtrees taking values $\Rig_1 \geq \Rig_2 \geq \cdots \geq \Rig_k \geq 0$, we have 
    \begin{equation}\label{eq:rigid}
        \Rig(T) \ceq 
        \begin{cases} 
        \Rig_1 = \max_{1 \leq i \leq k} \Rig_i  & \text{ if } k = 1, \\
        \Rig_1 + \ind{\Rig_1 = \cdots = \Rig_k}  & \text{ if } k > 1.
        \end{cases}
    \end{equation}
\end{enumerate}
Note that all these definitions coincide for binary trees. 

\begin{figure}[hbtp]
    \centering
    \begin{minipage}{.4\textwidth}
    \centering 
    \input{HS_french} \\ \smallskip \small 
    (i) French Horton-Strahler number
    \end{minipage} \quad 
    \begin{minipage}{.4\textwidth}
    \centering 
    \input{HS_canadian} \\ \smallskip \small 
    (ii) Canadian Horton-Strahler number
    \end{minipage} 
    \\ \medskip  
    
    \begin{minipage}{.4\textwidth}
    \centering 
    \input{HS_standard} \\ \smallskip \small 
    (iii) standard Horton-Strahler number
    \end{minipage} \quad 
    \begin{minipage}{.4\textwidth}
    \centering 
    \input{HS_rigid} \\ \smallskip \small 
    (iv) rigid Horton-Strahler number
    \end{minipage}
    
    \captionsetup{width=.6\textwidth}
    \caption{An illustration of the different Horton-Strahler numbers (i) -- (iv) for a given tree. In all cases, the leaves have value 0.}
    \label{fig:HSnumbers}
\end{figure}
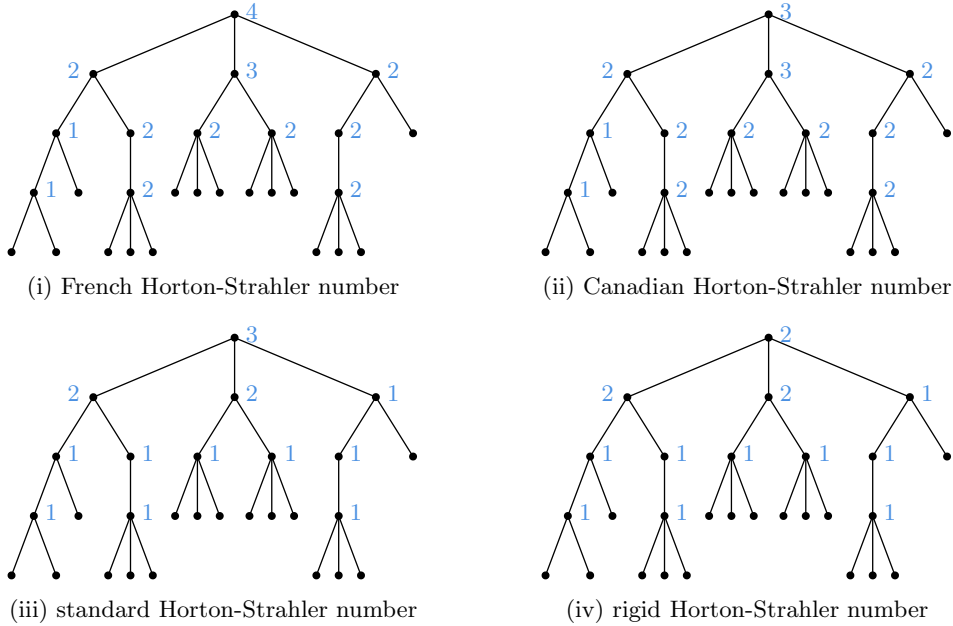

We also have the following ordering: 
\begin{lem}\label{lem:HSnumsordering}
For any tree $T$, the different Horton-Strahler numbers are ordered according to 
\begin{equation}\label{eq:HSnumsInequalities}
    \Fr(T) \geq \Can(T) \geq \HS(T) \geq \Rig(T).
\end{equation}
\end{lem}
The proof proceeds by induction on the height of the tree, and is given in Appendix~\ref{appendix:alternateHSproofs}. 

From this lemma, we immediately get that $(1/2) \log_2 n$ is a universal lower bound for both the French and the Canadian Horton-Strahler numbers $\Fr(T_n)$ and $\Can(T_n)$ of any critical conditional Galton-Watson tree $T_n$ with $0 < \sigma^2 < \infty$.
Indeed, the French Horton-Strahler number $\Fr(T_n)$ for a uniformly random $k$-ary tree $T_n$ of size $n$ was shown to satisfy 
\begin{equation}\label{eq:drmotaunif}
    \Fr(T_n) \sim \frac{1}{2} \log_2 n
\end{equation}
in probability by Drmota and Prodinger~\cite{drmota2006register}. They in fact show that $\Fr(T_n)$ is quite concentrated about $(1/2)\log_2 n$, regardless of the value of $k \geq 2$. We recall that a uniformly random $k$-ary tree of size $n$ is a conditional Galton-Watson tree with offspring $\xi \sim \mathrm{Binomial}(k, 1/k)$. Therefore, from what we have shown in this paper, its (standard) Horton-Strahler number also scales as $(1/2)\log_2 n$. One may then be tempted to believe that \eqref{eq:drmotaunif} holds for the French Horton-Strahler number of conditional Galton-Watson trees $T_n$ generated from any offspring distribution $\xi$ with finite variance $\sigma^2$, but that is false. The definition of $\Fr(T_n)$ is quite sensitive to the degree distribution: it is easy to see that if $M_n$ is the maximal degree of any node in $T_n$, then 
\begin{equation*}
    \Fr(T_n) \geq M_n - 1.
\end{equation*}

Maximal degrees of conditional Galton-Watson trees are well understood; see for example Janson's complete treatment~\cite{janson2012}. If $\xi$ has a polynomial tail, then the maximal degree $M_n$ grows at a polynomial rate as well. For exponential tails, $M_n$ grows as a constant multiple of $\log n$. Thus, for general critical offspring distributions, a $(1/2)\log_2 n$ upper bound for the French Horton-Strahler number does not hold. However, it seems plausible that for distributions with bounded degree or exhibiting a faster-than-exponential decrease in the tail, \eqref{eq:drmotaunif} would remain true.

\smallskip 

The Canadian Horton-Strahler number $\Can(T_n)$ is much less sensitive than $\Fr(T_n)$. Just like the French number, it satisfies the lower bound
\[\P \{\Can(T_n) \leq (1/2 - \e)\log_2 n \} = o(1) \]
for all $\e > 0$; but $\Can(T_n)$ can still be much larger than $(1/2)\log_2 n$.


\smallskip 

Finally, from Lemma~\ref{lem:HSnumsordering}, the rigid Horton-Strahler number has $(1/2)\log_2 n + o(1)$ as a strict upper bound. We can further study it using the tools developed in this paper. We will find that it tends as either $\log_2 \log_2 n$ or $\log_2 n$, modulo constant multiplicative factors. Our results are presented in section~\ref{sec:rigid}.

\medskip 

Another possible generalization of the Horton-Strahler can be given from the structural view of the number. We will recall the structural definition of the standard Horton-Strahler number (i.e., the register function) and define the $k$-ary register function for any tree $T$.
\begin{enumerate} 
    \item The \textit{register function} (i.e., the standard Horton-Strahler number) $\HS(T)$ is the height of the largest complete binary tree that can be embedded in $T$. 
    \item Similarly, we define the \textit{$k$-ary register function} $\K(T)$ for any given $k \geq 2$ to be the height of the largest complete $k$-ary tree that can be embedded in $T$. The definition can also be written recursively. First, set the value of a leaf node $u$ with $|T[u]| = 1$ to be $0$. Then, if the root of the tree $T$ has $\ell \geq k$ children with values $\K_1 \geq \K_2 \geq \dots \K_\ell$ (sorted in decreasing order), the tree has $k$-ary register function 
    \begin{equation}\label{eq:karyregisterdef}
        \begin{aligned}
        \K(T) &\ceq \K_1 + \ind{\K_1 = \cdots \K_k} \\ 
        &= \max\{ \K_1, \K_k + 1 \}.
        \end{aligned}
    \end{equation}
    If the tree has $\ell < k$ children, then $\K(T) = \K_1$.
\end{enumerate}
Note that as stated in the introduction, the register function corresponds to $\HS(T) + 1$ in the literature (which amounts to letting the leaves have value 1). We omit this difference in our discussion for clarity of notation.

The definitions of the regular register function and the $k$-ary register function coincide for $k=2$. We also have that $\mathrm{K}(T) \leq \HS(T)$ for any $k$. However, $\K(T)$ does not fit cleanly into the chain of inequalities in Lemma~\ref{lem:HSnumsordering}; its relationship to the rigid Horton-Strahler number depends on the specific offspring distribution. 

The asymptotic behaviour of the $k$-ary register function for a conditional Galton-Watson tree can be determined quite simply using the tools developed in this paper. The result will be presented in section~\ref{sec:karyregister}. We prove a lemma regarding the unconditional tree, and then the theorem follows by the same proof as for the rigid Horton-Strahler number.

\section{The Rigid Horton-Strahler Number}
\label{sec:rigid}
\no 
We begin with analogs of Lemma~\ref{lem:removep1} and Theorem~\ref{thm:uncond} regarding unconditional Galton-Watson trees for the rigid Horton-Strahler number. Note that we only need to deal with trees satisfying $\P\{\xi > 2\} > 0$, since all the definitions of the Horton-Strahler number coincide for binary trees. 

\begin{lem}\label{lem:removep1rigid}
Let $\xi$ be an offspring distribution with $\mu=1$ and $0 < \sigma^2 < \infty$. Consider the altered distribution $\zeta$ defined in Lemma~\ref{lem:removep1} with the probability of one child set to zero. Then, defining $\Rig$ and $\Rig'$ to be respectively the Horton-Strahler number of an unconditional $\xi$- and $\zeta$-Galton-Watson tree, we have 
\begin{equation}
    \Rig' \disteq \Rig.
\end{equation}
\end{lem}
This lemma is once again proved via induction, with details laid out in Appendix~\ref{appendix:rigid}. It is used to show the following analog of Theorem~\ref{thm:uncond} for the rigid Horton-Strahler number.

\begin{thm}\label{thm:uncondrigid}
Consider an unconditional critical Galton-Watson tree $T$ with $0 < \sigma^2 < \infty$. Define a parameter 
\begin{equation}\label{eq:definitiond}
    d \ceq \min \{ i > 1: p_i > 0 \}.
\end{equation}
If $d=2$, then 
\begin{equation}\label{eq:rigidthmCASE1}
    \P\{\Rig(T) = x\} = \bpar{1 + \sqrt{\frac{\sigma^2}{2p_2}}}^{-x + o(x)}.
\end{equation}
Otherwise, if $d > 2$, then there exist constants $\alpha_i > 0$ such that 
\begin{equation}\label{eq:rigidthmCASE2}
    (1 + \alpha_3)^{-\alpha_4 (d/2)^x} \leq \P\{\Rig(T) = x\} \leq (1 + \alpha_1)^{-\alpha_2 (d/2)^x}
\end{equation}
for $x \geq \alpha_5$.
\end{thm}
The proof of this theorem proceeds similarly to that of Theorem~\ref{thm:uncond}, and is included in Appendix~\ref{appendix:rigid}.
Note that for binary critical trees $T$, we have $p_0 = p_2$, implying $p_1 = 1 - 2p_2$, and $\sigma^2 = p_1 + 4p_2 - 1 = 2p_2$. 
Therefore, $\sqrt{\sigma^2/2p_2} = 1$ and, as expected, the rigid Horton-Strahler number 
is equal to the regular Horton-Strahler number: 
\[ \P \{\Rig(T) = x\} = 2^{-x + o(x)}. \]

We can now derive asymptotics for the rigorous Horton-Strahler number just as we did in sections~\ref{sec:lowerbound} and~\ref{sec:upperbound}. As shown in the preceding theorem, the parameter $d$ matters a lot, determining whether the growth scales as $\log_2 n$ or $\log_2 \log_2 n$. The results are formalized below.

\begin{thm}\label{thm:conditionalrigid}
Consider a critical Galton-Watson tree $T_n$ conditioned to be of size $|T| = n$, and define $d$ as in the previous theorem. If $d > 2$, we have 
\begin{equation}
    \frac{\Rig(T_n)}{\log_2 \log_2 n} \to \frac{1}{\log_2 d/2} 
\end{equation}
in probability as $n \to \infty$. On the other hand, if $d=2$, letting $\gamma = 1 + \sqrt{\sigma^2 / 2p_2}$, 
\begin{equation}
    \frac{\Rig(T_n)}{\log_2 n} \to \frac{1}{2 \log_2 \gamma}
\end{equation}
in probability as $n \to \infty$.
\end{thm}
\begin{proof}
Let us begin with the $d > 2$ case. The upper bound can be proven very simply. 
\begin{align*}
    \P\{\Rig(T) \geq x \mid |T| = n\} &= \frac{\P\{\Rig(T) \geq x, |T| = n\}}{\P\{|T| = n\}}  \\ 
    &\leq \frac{\P\{\Rig(T) \geq x\}}{\P\{|T| = n\}} \\ 
    &= \Theta(n^{3/2}) \P\{\Rig(T) \geq x\}
\end{align*}
where $T$ is an unconditional Galton-Watson tree. We can then bound it using Theorem~\ref{thm:uncondrigid}: there exist constants $\alpha_i > 0$ such that
\begin{equation*}
    \begin{aligned}
        \P\{\Rig(T) \geq x \mid |T| = n\} &\leq \Theta(n^{3/2}) \alpha_1 (1 + \alpha_2)^{-\alpha_3 (d/2)^x}.
    \end{aligned}
\end{equation*}
This tends to zero for $x = (1+\e) \frac{\log_2 \log_2 n}{\log_2 d/2}$.
\smallskip 

The lower bound can be proven following the outline of the ``intuitive proof" from section~\ref{sec:lowerbound}, using the same method as  Theorem~\ref{thm:lowerbound}.
We have the same decomposition as in \eqref{eq:lowerbounddecomposition}:
\begin{equation}
    \P\{\Rig(T) \leq x \mid |T| = n\} = \I + \II + \III,
\end{equation}
where $\I$, $\II$ and $\III$ are exactly as defined in \eqref{eq:lowerboundI-II-III}, except with $\HS$'s switched for $\Rig$'s in the definition of the third term. We showed in \eqref{eq:lbI} and \eqref{eq:lbII} that both $\I$ and $\II$ are $o(1)$. To upper bound $\III$, we can once again consider the truncated Kesten limit tree $T^\infty_\ell$ at $\ell = \sqrt{n}/\log n$ depicted in Figure~\ref{fig:T-inf-ell}, with unconditional hanging trees $T_{ij}$ i.i.d.\ distributed as $T$. Recall from \eqref{eq:numtrees} that the number of hanging trees $N$ satisfies 
\[\P\bcurly{N < (\ell + 1)\frac{\sigma^2}{2} } = o(1).\]
We can thus bound 
\begin{align*}
    \III &\leq \P \{ \Rig(T_\ell^\infty) \leq x \} \\ 
    &\leq \P\Big\{\max_{ij}  \Rig(T_{ij}) \leq x, N \geq (\ell + 1)\frac{\sigma^2}{2} \Big\} + \P\Big\{\max_{ij}  \HS(T_{ij}) \leq x, N < (\ell + 1)\frac{\sigma^2}{2} \Big\} \\ 
    &\leq (1 - \P\{\Rig(T) > x\})^{(\ell+1)\sigma^2/2} \\ 
    &\leq \exp \bpar{- \frac{\sigma^2}{2} \frac{\sqrt{n}}{\log n} \P \{\Rig(T) = x\} } \\ 
    &\leq \exp \bpar{- \frac{\sigma^2}{2} \frac{\sqrt{n}}{\log n} \bpar{\frac{1}{1+\alpha}}^{\beta (d/2)^x}} 
\end{align*}
for some $\alpha, \beta > 0$. As we wished to show, this tends to zero for $x = (1- \e) \frac{\log_2 \log_2 n}{\log_2 d/2}$ for any $\e > 0$.

\bigskip  
For the $d = 2$ case, we note that the form of $\P\{\Rig(T) = x\}$ in \eqref{eq:rigidthmCASE1} is identical to that of $\P\{\HS(T)= x\}$, where the base of the exponent changes from $2$ to $\gamma$. The proofs of the upper and lower bound for the regular Horton-Strahler number thus translate to this case exactly. We have
\begin{equation*}
    \P\{\Rig(T) \leq x \mid |T| = n\} \leq \exp\bpar{ - \frac{\sqrt{n}}{\log n} \gamma^{-x + o(x)} },
\end{equation*}
which tends to zero for $x = \bpar{\frac{1}{2\log \gamma} - \e}\log_2 n$ for any $\e > 0$, completing the lower bound. For the upper bound, there exists $c$ such that
\begin{equation*}
    \P\{\Rig(T) \geq x \mid |T| = n\} \leq c \sqrt{n} \gamma^{-x+ o(x)},
\end{equation*}
which tends to zero for $x = \bpar{\frac{1}{2\log \gamma} + \e}\log_2 n$ for any $\e > 0$.
\end{proof}

\section{The ${k}$-ary Register Function}
\label{sec:karyregister}
\no 
The $k$-ary register function $\K(T)$ was defined in \eqref{eq:karyregisterdef} as the height of the largest complete $k$-ary tree that can be embedded in $T$. We can show that the $k$-ary register function of a critical Galton-Watson tree converges to $(\log_2 k/2)^{-1}\log_2 \log_2 n$ in probability.
Recall that the asymptotic behaviour of the rigid Horton-Strahler for the unconditional tree --- Theorem~\ref{thm:uncondrigid} --- was quite tedious to prove. In contrast, we present a relatively simple proof of the analogous result for the $k$-ary register function, albeit with an extra restriction on the moments of the offspring distribution.

\begin{thm}\label{thm:uncondkary}
Suppose $k \geq 3$. Let $\xi$ be such that $\E\{\xi\} = 1$, $\V\{\xi\} > 0$, $\E\{\xi^{k+1}\} < \infty$ and $\P\{\xi \geq k\} > 0$. Let $T$ be an unconditional Galton-Watson tree. Then for some $x^*$ large enough, there exist $\alpha, \alpha' > 0$ and $\beta, \beta' \in (0,1)$ such that
\begin{equation}
    \alpha' \cdot \beta'^{(k/2)^x} \leq \P\{\K(T) = x\} \leq \alpha \cdot \beta^{(k/2)^x}
\end{equation}
for all $x \geq x^*$.
\end{thm}
Before proving this theorem, note that this is exactly the same as the tail bounds of the rigid Horton-Strahler number when $d > 2$; see Theorem~\ref{thm:uncondrigid}, with $k$ taking the place of $d$. Therefore, with minor modifications, Theorem~\ref{thm:conditionalrigid} gives us the asymptotic behaviour of $\K(T)$ for a conditional Galton-Watson tree:

\begin{cor}
Let $k \geq 3$ and let $\xi$ be as specified in the previous theorem. Then, letting $T_n$ denote a conditional Galton-Watson tree of size $n$,
\begin{equation}
    \frac{\K(T_n)}{\log_2 \log_2 n} \to \frac{1}{\log_2 k/2} 
\end{equation}
as $n \to \infty$ in probability. 
\end{cor}

We can now proceed to the proof of the result about unconditional conditional Galton-Watson trees. 

\begin{proof}[Proof of Theorem~\ref{thm:uncondkary}] Let us begin by defining $q_x = \P\{\K(T) = x\}$, as well as $\{p_i\}$, $q_x^+$ and $q_x^-$ analogously to how they were defined in previous sections.
We can first solve 
\[q_0 = p_0 + p_1q_0 + p_2 q_0^2 + \cdots p_{k-1}q_0^{k-1}\]
for a finite value of $q_0$. 

Then, for $x > 0$, by multiple uses of the inclusion-exclusion formula, 
\begin{equation*}
    q_x \leq \E\bcurly{\binom{\xi}{1}q_x - \binom{\xi}{2}q_x^2 + \binom{\xi}{3}q_x^3} - \bpar{\E\bcurly{ \binom{\xi}{k} (q_x^+)^k - (k+1) \binom{\xi}{k+1}(q_x^+)^{k+1} }} + \E\bcurly{\binom{\xi}{k}q_{x-1}^k}
\end{equation*}
and 
\begin{equation*}
    q_x \geq \E\bcurly{\binom{\xi}{1}q_x - \binom{\xi}{2}q_x^2} - \E\bcurly{\binom{\xi}{k}(q_x^+)^k} + \E\bcurly{\binom{\xi}{k} q_{x-1}^k} - \E\bcurly{(k+1)\binom{\xi}{k+1}q_{x-1}^{k+1} }.
\end{equation*}
Noting that $\E\big\{\binom{\xi}{1} \big\} = 1$, $\E\big\{\binom{\xi}{2} \big\} = \sigma^2 / 2$ and for any $\ell \leq k$, $\E\big\{\binom{\xi}{\ell} \big\} \ceq \mu_\ell < \infty$,
we have for $k \geq 3$,
\[ \frac{\sigma^2}{2}q_x^2 - \mu_3 q_x^3 + \mu_k (q_x^+)^k - \mu_{k+1}(k+1) (q_x^+)^{k+1} \leq \mu_k q_{x-1}^k. \]
It is easy to see that $q_x \to 0$ as $x \to \infty$. Therefore, for any $\e > 0$, we can find $x^*$ such that for all $x \geq x^*$, $q_x \leq \e$ and $q_x^+ \leq (1+\e)q_x$. We then have
\[\bpar{\frac{\sigma^2}{2} - \mu_3 \e -  (1+\e)\mu_{k+1} (k+1) (\e')^{k+1}} q_x^2 \leq \mu_k q_{x-1}^k\]
for $x \geq x^*$. Picking an $\e>0$ and corresponding $x^*$ such that the term in parentheses belongs to $[\sigma^2/3,  \sigma^2/2)$, for every $x \geq x^*$, 
\[q_x^2 \leq \frac{3\mu_k}{\sigma^2}q_{x-1}^k,\]
i.e., 
\begin{equation}\label{eq:karyupperb}
    q_x \leq \sqrt{\frac{3\mu_k}{\sigma^2}} q_{x-1}^{k/2}.
\end{equation}

Similarly, we have for the lower bound
\[ \frac{\sigma^2}{2} q_x^2 \geq - \mu_k (q_x^+)^k + \mu_k q_{x-1}^k - (k+1) \mu_{k+1} q_{x-1}^{k+1}, \]
from which we deduce for $\e>0$ and corresponding $x^{**}$ chosen such that for all $x \geq x^{**}$, $\mu_k (1+\e)^k \e^{k-2} \leq \sigma^2 / 2$ and $(k+1) \mu_{k+1} \e \leq \frac{\mu_k}{2}$, 
\begin{equation*}
\begin{aligned}
    \sigma^2 q_x^2 &\geq \frac{\sigma^2}{2}q_x^2 + \mu_k (1+\e)^k \e^{k-2} q_x^2  \\ 
    &\geq \frac{\sigma^2}{2}q_x^2 + \mu_k (1+\e)^k q_x^k \\
    &\geq \frac{\sigma^2}{2}q_x^2 + \mu_k (q_x^+)^k \\
    &\geq \frac{\mu_k}{2}q_{x-1}^k.
\end{aligned}
\end{equation*}
Therefore
\begin{equation}\label{eq:karylowerb}
    q_x \geq \sqrt{\frac{\mu_k}{2\sigma^2}} q_{x-1}^{k/2}
\end{equation}
for all $x \geq x^{**}$.
The theorem statement is obtained by taking $x \geq \max\{x^*, x^{**}\}$ and combining the two estimates \eqref{eq:karyupperb} and \eqref{eq:karylowerb}.
\end{proof}

\paragraph{A related result} Cai and Devroye showed that the height $H_n$ of the maximal complete $k$-ary tree occurring as a terminal element in a critical Galton-Watson tree $T_n$ satisfies 
\[\frac{H_n}{\log_2 \log_2 n} \to \frac{1}{\log_2 k}\]
in probability (see Lemma 4.2, \cite{Cai2017}). These elements are called fringe subtrees. They also showed the same behaviour for the height $H_n'$ of the maximal complete $k$-ary non-fringe tree which is allowed to occur as a non-terminal element in $T_n$ (see Lemma 5.7, \cite{Cai2017}).

In this paper, we allow the complete $k$-ary tree to be embedded in $T_n$ rather than an element of it. We show that asymptotically, the height of the root is still a constant factor of $\log_2 \log_2 n$. The constant is now larger than in the case analyzed by Cai and Devroye: $(\log_2 k/2)^{-1}$ rather than $(\log_2 k)^{-1}$. 



\section*{Conclusion and Future Work}
\no 
In this work, we considered the setting of critical conditional Galton-Watson trees. We showed that their Horton-Strahler number scales as $\Theta(\log_2 n)$ in probability. This result was proven using the convergence of a conditional Galton-Watson tree to Kesten's limit tree, as well as the construction of a rotationally invariant event using the random walk view of a tree.

We then defined several other generalizations of the Horton-Strahler number to non-binary trees, including the rigid Horton-Strahler number and the $k$-ary register function. For the rigid Horton-Strahler case, we identify a key parameter $d$ denoting the first integer $i \geq 2$ for which the offspring distribution has nonzero probability of having $i$ children. We then used the same methods introduced earlier in the paper to prove that the $k$-ary register function and the rigid Horton-Strahler number both scale as $\Theta(\log_2 \log_2 n)$, respectively when $k \geq 3$ and $d \geq 3$. 

Our main result from sections~\ref{sec:lowerbound} and \ref{sec:upperbound} generalizes all previously known first order results for the regular Horton-Strahler number. However, higher order concentration information is not presented here. It seems plausible that the variance of $\HS(T_n)$ is $\O(1)$; such a result would be very desirable.

\section*{Acknowledgements}
\no 
Luc Devroye's research was supported by a Discovery Grant from NSERC. We would like to thank Konrad Anand, Marcel Goh, Jad Hamdan, Tyler Kastner, Gavin McCracken, Ndiam{\'e} Ndiaye, and Rosie Zhao for moral support, feedback and enlightening discussions. Special thanks to Marcel Goh for expert knowledge of the \TeX book.

\bibliographystyle{abbrv}
\bibliography{biblio}


\appendix

\section{Proofs for Unconditional Galton-Watson Trees}
\label{appendix:uncond}

\begin{proof}[Proof of Lemma~\ref{lem:removep1}] We will show that $q_i' \equiv \P\{\HS' = i\} = q_i$ for all $i \in \N$ via induction on $i$. \smallskip 

\no \textit{First, the base case}: for the $\zeta$ distribution, since $\P\{\zeta = 1\} = 0$, if the root has a non-zero number of children, it automatically has two or more children and a Horton-Strahler number greater or equal to one. Therefore
\[ q_0' = \P\{\zeta = 0\} = \frac{p_0}{1-p_1}. \]
For the $\xi$ distribution, we have $p_1 \neq 0$. Thus, for the root to have a Horton-Strahler number of zero, the tree must be a path, yielding
\[ q_0 = p_0 + p_1 p_0 + p_1^2 p_0 + \cdots = \frac{p_0}{1-p_1}.\]

\no \textit{Now suppose that $q_j' = q_j$ for all $j < i$.} For the root to have Horton-Strahler number $i > 0$, either it has one child with this same number, or it has $\ell \geq 2$ children. In the second case, there are two further possibilities, where either the Horton-Strahler number does not change from the maximal number of the root's children, or it increases by one, with $r \geq 2$ children with Horton-Strahler number $i-1$. 
We define a function $\psi(q_i, q_{i-1}^-, q_{i-1}, q_{i-2}^-, \{p_\ell\})$ to encapsulate the probability of the root having Horton-Strahler number $i$ given these two possibilities: 
\[\psi(q_i, q_{i-1}^-, q_{i-1}, q_{i-2}^-, \{p_\ell\}) =  \sum_{\ell=2}^\infty p_\ell \bigg(\ell q_i (q_{i-1}^-)^{\ell - 1} + \sum_{r = 2}^\ell \binom{\ell}{r} q_{i-1}^r (q_{i-2}^-)^{\ell - r} \bigg).\]
Then,
\begin{equation}
    q_i = p_1 q_i + \psi(q_i, q_{i-1}^-, q_{i-1}, q_{i-2}^-, \{p_\ell\}) 
\end{equation}
and since $q_i'$ has $\Pr\{\zeta = 1\} = 0$,
\begin{align}
    q_i' &= \psi\big(q_i', q_{i-1}'^-, q_{i-1}'^-, q_{i-2}', \{p_\ell/(1-p_1)\}\big)  \label{eq:probrecursion} \\ 
    &= \psi\big(q_i', q_{i-1}^-, q_{i-1}, q_{i-2}, \{p_\ell/(1-p_1)\}\big) \nonumber
\end{align}
where the second line was obtained from the inductive hypothesis. Thus, 
\[q_i\Big(1 - p_1 - \sum_{\ell=2}^\infty p_\ell \ell (q_{i-1}^-)^{\ell - 1} \Big) = \sum_{\ell=2}^\infty p_\ell  \sum_{r = 2}^\ell \binom{\ell}{r} q_{i-1}^r (q_{i-2}^-)^{\ell - r} \ceq \theta\]
and 
\[q_i' \Big(1 - \sum_{\ell=2}^\infty \frac{p_\ell}{1-p_1} \ell (q_{i-1}^-)^{\ell - 1} \Big) =  \theta \frac{1}{1-p}; \]
we have shown that $q_i' = q_i$. By induction, this thus holds for all $i \in \N$.
\end{proof}

\begin{proof}[Proof of Theorem~\ref{thm:uncond}]
By Theorem~\ref{lem:removep1}, we can without loss of generality assume that $p_1 = 0$. We have a recursion from \eqref{eq:probrecursion}: 
\[q_i = \sum_{\ell=2} p_\ell \bigg(\ell q_i (q_{i-1}^-)^{\ell - 1} + \sum_{r = 2}^\ell \binom{\ell}{r} q_{i-1}^r (q_{i-2}^-)^{\ell - r} \bigg).\]
Rearranging gives us
\[q_i\Big(1 - p_1 - \sum_{\ell=2}^\infty p_\ell \ell (q_{i-1}^-)^{\ell - 1} \Big) = \sum_{\ell=2}^\infty p_\ell  \bigg( \sum_{r = 0}^\ell \binom{\ell}{r} q_{i-1}^r (q_{i-2}^-)^{\ell - r} - \ell q_{i-1}^1 (q_{i-2}^-)^{\ell - 1} - (q_{i-2}^-)^\ell \bigg) \]
and we can use the binomial theorem and the definition of the generating function $f(s)$ of $\{p_i\}$ to obtain 
\begin{align*}
    q_i (1 - f'(q_{i-1}^-)) &= \sum_{\ell=2} p_\ell \bpar{(q_{i-1} + q_{i-2}^-)^\ell - \ell q_{i-1} (q_{i-2}^-)^{\ell - 1} - (q_{i-2}^-)^\ell } \\ 
    &= \sum_{\ell=2}^\infty p_\ell (q_{i-1}^-)^\ell - q_{i-1} \sum_{\ell=2}^\infty \ell p_\ell (q_{i-2}^-)^{\ell-1} - \sum_{\ell=2}^\infty p_\ell q_{i-2}^{-\ell} \\ 
    &= \left( f(q_{i-1}^-) - p_0 \right) - q_{i-1} f'(q_{i-2}^-) - \left( f(q_{i-2}^-) - p_0 \right) \\ 
    &= \left( f(q_{i-1}^-) - f(q_{i-2}^-) \right) - q_{i-1} f'(q_{i-2}^-).
\end{align*}
We thus have 
\begin{equation}\label{eq:qirecursiongenerating}
    q_i = \frac{ f(q_{i-1}^-) - f(q_{i-2}^-) - q_{i-1} f'(q_{i-2}^-)}{1 - f'(q_{i-1}^-)}.
\end{equation}
Now, consider the Taylor expansion of $f(s)$ near $s=1$. Then, 
\[f(s) = 1 + \alpha_1 (s-1) + \alpha_2 \frac{(s-1)^2}{2!} + \cdots,\]
where $\alpha_i$ is the $i$-th descending moment of $\xi$
\[\alpha_i = \E\{\xi(\xi-1)\cdots (\xi-i+1)\}.\]
In particular, we have $\alpha_1 = 1$ and $\alpha_2 = \sigma^2$.

Also, by Taylor's series with remainder, for some $\e \in [0,1],$ 
\[f(q_{i-1}^-) = f(q_{i-2}^-) + q_{i-1} f'(q_{i-2}^-) + \frac{q_{i-1}^2}{2} f''(q_{i-2}^- + \e q_{i-1}),\]
and \eqref{eq:qirecursiongenerating} becomes 
\begin{equation*}
    q_i = \frac{(q_{i-1}^2 / 2) f''(q_{i-2}^- + \e q_{i-1}) }{1 - f'(q_{i-1}}.
\end{equation*}
We further have, for some $\e' \in [0,1]$, 
\begin{align*}
    f'(q_{i-1}^-) &= f'(1) + (q_{i-1}^- - 1) f''(q_{i-1}^- + \e' q_i^+) \\ 
    &= 1 + (q_{i-1}^- -1) f''(q_{i-1}^- + \e' q_i^+).
\end{align*}
Thus, 
\[q_i = \frac{q_{i-1}^2}{2(1-q_{i-1}^-)} \cdot \frac{f''(q_{i-2}^- + \e q_{i-1})}{f''(q_{i-1}^- + \e' q_i^+)}.\]
Since $f''(s)$ is an increasing function, we have the inequalities 
\begin{equation*}
    \frac{q_{i-1}^2}{2q_i^+} \cdot \frac{f''(q_{i-2}^-)}{f''(1)} \leq q_i \leq \frac{q_{i-1}^2}{2q_i^+} \cdot \frac{f''(q_{i-1}^- )}{f''(q_{i-1}^-)} =  \frac{q_{i-1}^2}{2q_i^+}, 
\end{equation*}
where the ratio $f''(q_{i-2}^-) / f''(1)$ is near $1$ since $f''(1) = \sigma^2 \in (0, \infty)$ and $q_{i-2}^- \to 1$ as $i \to \infty$. Thus, for every $\e > 0$, there is some $n_0(\e)$, such that for all $i \geq n_0(\e)$, 
\[(1-\e) \frac{q_{i-1}^2}{2q_i^+} \leq q_i.\]
We thus have for all $i \geq n_0(\e)$,
\begin{equation}\label{eq:qiinequality}
    (1-\e) \frac{q_{i-1}^2}{2q_i^+} \leq q_i \leq \frac{q_{i-1}^2}{2q_i^+}.
\end{equation}
In the following, we will set $\e > 0$ and consider $i \geq n_0(\e)$. 
First, note that from \eqref{eq:qiinequality}, $q_i^2 \leq q_i q_i^+ \leq q_{i-1}^2/2$, so 
\[q_i \leq q_{i-1} / \sqrt{2}.\]
Our result will follow from the fact that if we have $q_i \leq q_{i-1}\cdot \gamma$ for some $\gamma < 1$, then 
\begin{equation}\label{eq:qiratiotobeproven}
\frac{1}{2} - h(\e) \leq \frac{q_i}{q_{i-1}} \leq \frac{1}{2} + g(\e)
\end{equation}
for positive functions $h$ and $g$ with $\lim_{\e \to 0} h(\e) = \lim_{\e\to 0} g(\e) = 0$. This will give us that, for $i \geq n_0(\e)$, 
\begin{equation*}
    q_{n_0(\e)} \bpar{\frac{1}{2} - h(\e)}^{i-n_0(\e)} \leq q_i \leq q_{n_0(\e)} \bpar{\frac{1}{2} + g(\e)}^{i-n_0(\e)},
\end{equation*}
which completes the proof, as $\e> 0$ was chosen arbitrarily. 

\medskip 
Let's now show \eqref{eq:qiratiotobeproven}. From $q_i \leq q_{i-1}\gamma$, we have that 
\[q_i^+ \leq q_i(1 + \gamma + \gamma^2 + \cdots) = \frac{q_i}{1-\gamma}.\]
Then, we have 
\[q_i \geq (1-\e) \frac{q_{i-1}^2}{2q_i^+} \geq (1-\e)(1-\gamma) \frac{q_{i-1}^2}{2q_i},\]
giving us 
\begin{equation}\label{eq:qiratiolowerbound}
    q_i \geq q_{i-1} \sqrt{ \frac{(1-\e)(1-\gamma)}{2}} .
\end{equation}
For the upper bound, we have that
\[q_i^+ \geq q_i \bpar{ 1 + \sqrt{ \frac{(1-\e)(1-\gamma)}{2}} + \cdots } = \frac{q_i}{1 - \sqrt{ \frac{(1-\e)(1-\gamma)}{2}}}, \]
and thus 
\[q_i \leq \frac{q_{i-1}^2}{2q_i^+} \leq \frac{q_{i-1}^2}{2q_i}\bpar{1 - \sqrt{ \frac{(1-\e)(1-\gamma)}{2}}}.\]
Similarly to the lower bound, this gives us 
\[q_i \leq q_{i-1} \sqrt{\frac{1 - \sqrt{(1-\e)(1-\gamma)/2} }{2}}.\]
Now consider the map $\gamma \mapsto \sqrt{\frac{1 - \sqrt{(1-\gamma)/2}}{2}}$; it has a fixed point at $\gamma = 1/2$ since $\sqrt{(1 - \sqrt{1/4})/2} = 1/2$. More precisely, let $\gamma$  be the solution of 
\[\gamma = \sqrt{\frac{1 - \sqrt{(1-\e)(1 - \gamma)/2}}{2}}.\]
Then $\gamma = 1/2 + g(\e)$ for some $g(\e) > 0$, $g(\e) \to 0$ as $\e \to 0$. Therefore, recalling the lower bound \eqref{eq:qiratiolowerbound}, we have for all $i \geq n_0(\e)$, 
\[\sqrt{\frac{(1-\e)(1-\gamma)}{2}} = \frac{1}{2}\sqrt{(1-\e)(1 - 2g(\e))} \leq \frac{q_i}{q_{i-1}} \leq \frac{1}{2} + g(\e),\]
and we have shown \eqref{eq:qiratiotobeproven}. 
\end{proof}

\section{Proofs for Alternate Horton-Strahler Numbers}
\label{appendix:alternateHSproofs}
\begin{proof}[Proof of Lemma~\ref{lem:HSnumsordering}] 
We proceed by induction on the height of the tree to show \eqref{eq:HSnumsInequalities}. Consider a tree $T$ with $k$ children, and consider all the required orderings of the French, Canadian, standard and rigid Horton-Strahler numbers ($\Fr_i,\ \Can_i,\ \HS_i$ and $\Rig_i$ for $i = 1, \dots, k$) of these children. Note that for a leaf node with subtree size $|T| = 1$, the base case holds: $\Fr(T) = \Can(T) = \HS(T) = \Rig(T) = 0$.
\begin{enumerate} 
    \item To show $\Fr(T) \geq \Can(T)$, suppose that for each $1 \leq i \leq k$ children, $\Can_i \leq \Fr_i$. Suppose $\Can_1 = \cdots = \Can_r$ for some $r \in \{1, \dots, k\}$. Then, $\Can = \Can_r + (r - 1) \leq \Fr_r + (r-1) \leq \Fr$, and we are done.
    
    \item To show $\Can(T) \geq \HS(T)$, suppose that for each $1 \leq i \leq k$ children, $\HS_i \leq \Can_i$. Then, in the case where $\Can(T) > \Can_1$, we are done, as $\HS(T) \leq \HS_1 + 1 \leq \Can_1 + 1 \leq \Can(T)$. Otherwise, $\Can(T) = \Can_1$ and $r = 1$, so we have the strict ordering $\Can_1 > \Can_2 \geq \cdots$. This leads to the two cases:
    \begin{itemize}
        \item if $\HS_1 < \Can_1$, then we are again done since $\HS(T) \leq \HS_1 + 1$.
        \item otherwise, $\HS_1 = \Can_1 > \Can_2 \geq \HS_2$, so $\HS(T) = \HS_1 \leq \Can_1 = \Can(T)$, as required.
    \end{itemize}
    
    \item To show $\HS(T) \geq \Rig(T)$, suppose that for each $1 \leq i \leq k$ children, $\Rig_i \leq \HS_i$. We can proceed the same way as in (ii). If $\HS(T) > \HS_1$, then we are done, as $\Rig(T) \leq \Rig_1 + 1 \leq \HS_1 + 1 = \HS(T)$. Otherwise, $\HS(T) = \HS_1$ and we have $\HS_1 > \HS_2 \geq \cdots$; there are two cases:
    \begin{enumerate}
        \item if $\Rig_1 < \HS_1$, then we are done, as $\Rig(T) \leq \Rig_1 + 1$.
        \item otherwise, $\Rig_1 = \HS_1 > \HS_2 \geq \Rig_2$, and thus $\Rig(T) = \Rig_1 = \HS_1 = \HS(T)$.
    \end{enumerate}
\end{enumerate}
All of these were shown at the root. Thus, the inequality holds by induction.
\end{proof}

\section{Proofs for the Rigid Horton-Strahler Number}\label{appendix:rigid}
\begin{proof}[Proof of Lemma~\ref{lem:removep1rigid}] We prove this lemma by induction. Define $q_i$ and $q_i'$ as 
\begin{align*}
    q_i &= \P\{\Rig(T) = i\} \\ 
    q_i' &= \P\{\Rig'(T) = i\},
\end{align*}
and recall the definitions of $q_i^-$ and $q_i^+$ offered at the start of section~\ref{sec:unconditional}.
Further recall that we defined $p_i = \P\{\xi = i\}$ and $p_i' = \P\{\zeta = i\}$, with $p_1' = 0$ and
\[p_i' = \frac{p_i}{1-p_1}\]
for $i \neq 1$.
\smallskip 

\no \textit{The base case}: we have by the same argument as in the proof of Lemma~\ref{lem:removep1} that 
\[q_0' = \frac{p_0}{1-p_1}\] 
and 
\[q_0 = p_0 + p_1p_0 + p_1^2p_0 + \dots = \frac{p_0}{1-p_1}.\]

\no \textit{Then suppose for induction that $q_j' = q_j$ for all $j < i$.} For the root to have Horton-Strahler number $i > 0$, either it has one child with the same number, or $\ell > 2$ children. In this second case, either all children have the same Horton-Strahler number $i-1$, or some number $r \in \{1, \ell - 1\}$ of children have Horton-Strahler number $i$. Defining $\psi(q_i, q_{i-1}, q_{i-1}^-, \{p_\ell\})$ as 
\begin{equation*}
    \psi(q_i, q_{i-1}, q_{i-1}^-, \{p_\ell\}) = \sum_{\ell = 2}^\infty p_\ell \Big( q_{i-1}^\ell + \sum_{r = 1}^{\ell - 1} \binom{\ell}{r} q_i^r (q_{i-1}^-)^{\ell - r} \Big),
\end{equation*}
we have 
\begin{equation}
    q_i = p_1 q_i + \psi(q_i, q_{i-1}, q_{i-1}^-, \{p_\ell\}).
\end{equation}
Then, we have that for the modified distribution,
\begin{equation*}
    \begin{aligned}
    q_i' &= \psi(q_i', q_{i-1}', (q_{i-1}^-)', \{p_i'\})\\ 
    &= \psi(q_i, q_{i-1}, q_{i-1}^-, \{p_\ell'\})
    \end{aligned}
\end{equation*}
from the inductive hypothesis. Then, we note that since only terms $p_i$ for $i \geq 2$ are involved, 
\begin{align*}
   q_i' &= \frac{1}{1-p_1} \psi(q_i, q_{i-1}, q_{i-1}^-, \{p_\ell\}) \\ 
   &= q_i,
\end{align*}
completing the proof.
\end{proof}

\begin{proof}[Proof of Theorem~\ref{thm:uncondrigid}]
Let $q_k = \P\{\Rig(T) = k\}$. We once again assume by Lemma~\ref{lem:removep1rigid} that $p_1 = 0$. Note that $\sigma^2$ will be involved in the proof and the results, and when the offspring distribution is changed from $\xi$ to $\zeta$ as in Lemma~\ref{lem:removep1rigid}, the standard deviation changes by a factor of $(1-p_1)^{-1}$. 
However, we will find that the final form of the result is such that this change in distribution does not matter. 

Using the same notation as in the proof of Theorem~\ref{thm:uncond}, we have that
\[q_0 = \P\{\R(T) = 0\} = p_0\] 
and \vspace{-.5\bs} 
\begin{align}
    q_k &= \sum_{\ell = d}^\infty p_\ell \Big( q_{k-1}^\ell + \sum_{r = 1}^{\ell - 1} \binom{\ell}{r} q_k^r (q_{k-1}^-)^{\ell - r} \Big) \nonumber \\ 
    &= (f(q_{k-1}) - p_0) + \sum_{\ell = d}^\infty p_\ell \bpar{(q_k + q_{k-1}^-)^\ell - q_k^\ell - (q_{k-1}^-)^\ell } \nonumber \\ 
    &= f(q_{k-1}) - f(q_k) + f(q_k^-) - f(q_{k-1}^-). \label{eq:rigidstarteq}
\end{align}
By the Taylor series with remainder, for some $\theta, \theta', \theta'' \in [0,1]$, we have approximations of the terms in \eqref{eq:rigidstarteq}:
\begin{equation}\label{eq:rigidtaylorapprox}
\begin{aligned}
    f(q_k) &= p_0 + \frac{1}{d!} q_k^d f^{(d)}(\theta' q_k), \\ 
    f(q_{k-1}) &= p_0 + \frac{1}{d!} q_{k-1}^d f^{(d)}(\theta'' q_k) \\
    f(q_k^-) - f(q_{k-1}^-) &= q_k f'(q_{k-1}^-) + \frac{q_k^2}{2!} f''(q_{k-1}^- + \theta q_k).
\end{aligned}
\end{equation}
Recall that $q_k \to 0$ as $k\to \infty$ and $f^{(i)} (0) = i! p_i$ for all $i \geq 0$. Then, since $f$ and all its derivatives are continuous, increasing and convex on $[0,1]$, for any $\e > 0$, there is some $n_0(\e)$ such that for all $k \geq n_0(\e)$, for all $r \geq 1$,
\begin{equation*}
    p_r \leq \frac{1}{r!} f^{(r)}(q_k) \leq p_r (1+\e).
\end{equation*}
Furthermore, since $f'(1) = 1$ and $f''(1) = \sigma^2$, we also have
\begin{align*}
    1 - \e &\leq f'(q_{k-1}) \leq 1 \\ 
    \sigma^2 (1-\e) &\leq \ f''(q_{k}^-) \ \leq \sigma^2. 
\end{align*}
These two facts can be used to simplify respectively the first two and the third equations in \eqref{eq:rigidtaylorapprox}. Then, plugging the terms back into our original form \eqref{eq:rigidstarteq} gives an upper bound for all $k \geq n_0(\e)$ of
\begin{equation}
    q_k \leq \frac{(q_{k-1}^d - q_k^d )p_d + \e q_{k-1}^d p_d + q_k^2 \sigma^2 / 2 }{1 - f'(q_{k-1}^-)}.
\end{equation}
Furthermore, 
\[ f'(q_{k-1}^-) = f'(1) - (1-q_{k-1}^-) f''(1-\de q_k^+)\]
for some $\de \in [0,1]$, yielding 
\[ 1 - q_k^+ \sigma^2 \leq f'(q_{k-1}^-) \leq 1 - q_k^+ \sigma^2 (1-\e).\]
We thus have 
\begin{equation}\label{eq:rigidupperbound}
    \begin{aligned}
        q_k &\leq  \frac{(q_{k-1}^d - q_k^d) p_d}{q_k^+ \sigma^2 (1-\e)} +  \frac{\e q_{k-1}^d p_d}{q_k^+ \sigma^2 (1-\e)} + \frac{q_k^2}{2q_k^+ (1-\e)}\\ 
        &\leq \frac{(q_{k-1}^d - q_k^d) p_d}{q_k \sigma^2 (1-\e)} +  \frac{\e q_{k-1}^d p_d}{q_k \sigma^2 (1-\e)} + \frac{q_k}{2(1-\e)} ,
    \end{aligned}
\end{equation}
which yields 
\[\frac{q_k^2 \sigma^2}{2} \leq q_{k-1}^d p_d \frac{1+\e}{1-\e}\]
and 
\begin{equation}
    q_k \leq \sqrt{\frac{2p_d}{\sigma^2}\bpar{\frac{1+\e}{1-\e}}} q_{k-1}^{d/2}.
\end{equation}

We must now distinguish between the two cases $d=2$ and $d\geq 3$ as stated in the theorem. We begin with the case $d \geq 3$. In this case since $\sum_i ip_i = 1$, $p_d \leq 1/d$ and $\sigma^2 \geq 1$, thus, 
\[ \frac{2p_d}{\sigma^2} \leq \frac{2}{d}.\] 
Nothing that $\e$ was arbitrary, we can pick $\e = 1/8$ such that for all $d \geq 3$, $k \geq n_0(1/8)$, 
\begin{equation}\label{eq:rigidupperboundresult}
    q_k \leq \sqrt{\frac{6}{7}} q_{k-1}^{d/2},
\end{equation}
and thus the upper bound follows for $d \geq 3$. 

For the lower bound in the $d \geq 3$ case, we can obtain from \eqref{eq:rigidstarteq} similarly to the upper bound case that 
\begin{equation}\label{eq:rigidlowerbound}
\begin{aligned}
    q_k &\geq \frac{(q_{k-1}^d - q_k^d)p_d - \e q_{k-1}^d p_d + q_k^2 \sigma^2 (1-\e)/2 }{1 - f'(q_{k-1}^-)} \\ 
    &\geq \frac{(1-\e)q_{k-1}^d p_d - q_{k-1}^d p_d + q_k^2 \sigma^2 (1-\e)/2 }{q_k^+ \sigma^2 }.
\end{aligned}
\end{equation}
Using \eqref{eq:rigidupperboundresult}, we can bound $q_k^dp_d$ by 
\[q_k^d p_d \leq \bpar{\frac{6}{7}}^{3/2} q_{k-1}^d p_d 
\]
and $q_k^+$ by
\[
    q_k^+ \leq q_k \Big(\sum_{n=0}^\infty \sqrt{6/7}^n \Big) = \frac{1}{1-\sqrt{6/7}}q_k < 14 q_k.
\]
These bounds give
\[ 
    q_k \geq \frac{\big(1 - \e - (6/7)^{3/2}\big) q_{k-1}^d p_d + q_k^2 \sigma^2 (1-\e)/2 }{14 \sigma^2 q_k} 
\]
i.e., 
\[
    q_k^2 \bpar{14 - \frac{1-\e}{2}} \sigma^2 \geq q_{k-1}^d p_d \bpar{1- \e - (6/7)^{3/2} }
\] 
and
\[
    q_k^2 \geq \bpar{\frac{2p_d}{\sigma^2} \frac{1- \e - (6/7)^{3/2}}{27 + \e}} q_{k-1}^d. 
\]
We can then choose $\e > 0$ such that
\[
    q_k^2 \geq \frac{2p_d}{\sigma^2} \frac{1}{162} q_{k-1}^d, 
\]
and thus, for all $k \geq n_0(\e)$,
\begin{equation}
    q_k \geq \frac{1}{9}\sqrt{\frac{p_d}{\sigma^2}} q_{k-1}^{d/2}.
\end{equation}
The lower bound follows from this. 
\smallskip

Finally, consider the case $d = 2$. From \eqref{eq:rigidupperbound}, setting $\alpha = 2p_2 / \sigma^2$, 
\[ 
    q_k \leq \frac{q_k^2}{2 q_k^2 (1 - \e) } + \frac{(q_{k-1}^2 - q_k^2) \alpha }{2 q_k^+ (1 - \e)} + \frac{\e q_{k-1}^2 \alpha }{2 q_k^+ (1-\e)}, 
\]
and thus 
\[ 
    2 q_k q_k^+ (1-\e) \leq q_k^2 (1-\alpha) + q_{k-1}^2 \alpha (1+\e). 
\]
Similarly, for the lower bound, we have from \eqref{eq:rigidlowerbound} that 
\[
    q_k \geq \frac{(q_{k-1}^2 - q_k^2) p_2 - \e q_{k-1}^2 p_2 + q_k^2 \sigma^2 (1-\e) / 2 }{q_k^+\sigma^2} 
\]
and thus 
\[
    2q_k q_k^+ \geq q_k^2 (1-\alpha - \e) + q_{k-1}^2 \alpha (1-\e).
\]
Letting $\gamma = q_k / q_{k-1}$, we know that $\gamma \leq 1$. We then have $q_{k-1}^2 = \gamma^{-2} q_k^2$ and $q_k^+ = (1-\gamma)^{-1} q_k$. Disregarding epsilons, the upper and lower bounds above coincide, giving
\[ 
    2q_k q_k^+ = \frac{2}{1-\gamma} q_k^2 =  (1-\alpha) q_k^2 + \frac{\alpha}{\gamma^2} q_k^2,
\]
which can be written as a simple equation
\begin{equation*}
    \frac{2}{1-\gamma} = 1 - \alpha + \frac{\alpha}{\gamma^2}
\end{equation*}
with solution $\gamma = \frac{\sqrt{\alpha}}{\sqrt{\alpha} + 1}$. Thus, similarly to~\eqref{eq:qiratiotobeproven} in the proof of Theorem~\ref{thm:uncond}, we have that 
\begin{equation}
    \frac{1}{1 - \sqrt{{\sigma^2}/{2p_2}}} - h(\e) \leq \frac{q_k}{q_{k-1}} \leq \frac{1}{1 - \sqrt{{\sigma^2}/{2p_2}}} + g(\e)
\end{equation}
for positive function $h$ and $g$ with $\lim_{\e \to 0} h(\e) = \lim_{\e \to 0} g(\e) = 0$. This proves the statement.
\medskip 

Note that it was safe to assume that $p_1 = 0$: when changing the distribution from $\xi$ to $\zeta$, both $\sigma^2$ and $p_2$ pick up a factor of $(1-p_1)^{-1}$, resulting in no net change in the ratio.
\end{proof}

\end{document}

%% file: cmds.tex
\newcommand{\E}{\mathbb{E}}
\newcommand{\I}{\mathbb{I}}

\newcommand{\N}{\mathbb{N}}
\renewcommand{\P}{\mathbb{P}}

\newcommand{\R}{\mathbb{R}}

\usepackage{bbm}
\newcommand{\1}{\mathbbm{1}}

\newcommand{\K}{\mathcal{K}}
\renewcommand{\L}{\mathcal{L}}

\renewcommand{\O}{\mathcal{O}}

\newcommand{\V}{\mathcal{V}}

\newcommand{\ad}{\text{ and }}

\newcommand{\ceq}{\coloneqq} 

\newcommand{\e}{\varepsilon}
\newcommand{\de}{\delta}

\DeclarePairedDelimiterX{\inp}[2]{\langle}{\rangle}{#1, #2} 

\newcommand{\bpar}[1]{\left(#1\right)}

\newcommand{\bcurly}[1]{\left\{#1\right\}}

\newcommand{\no}{\noindent}

\newcommand{\bs}{\baselineskip}

\usepackage{calligra}
    \DeclareMathAlphabet{\mathcalligra}{T1}{calligra}{m}{n}
    \DeclareFontShape{T1}{calligra}{m}{n}{<->s*[2.2]callig15}{}


%% file: T_completebinary.tex
\begin{tikzpicture}[x=0.75pt,y=0.75pt,yscale=-.75,xscale=.75]

\draw    (130,30) .. controls (141.8,43.45) and (129.3,44.45) .. (132.3,60.45) .. controls (135.3,76.45) and (114,72) .. (130,90) ;
\draw [shift={(130,90)}, rotate = 48.37] [color={rgb, 255:red, 0; green, 0; blue, 0 }  ][fill={rgb, 255:red, 0; green, 0; blue, 0 }  ][line width=0.75]      (0, 0) circle [x radius= 2.01, y radius= 2.01]   ;
\draw [shift={(130,30)}, rotate = 48.74] [color={rgb, 255:red, 0; green, 0; blue, 0 }  ][fill={rgb, 255:red, 0; green, 0; blue, 0 }  ][line width=0.75]      (0, 0) circle [x radius= 2.01, y radius= 2.01]   ;
\draw    (130,90) .. controls (132.04,102.56) and (87,104) .. (90,120) ;
\draw [shift={(90,120)}, rotate = 79.38] [color={rgb, 255:red, 0; green, 0; blue, 0 }  ][fill={rgb, 255:red, 0; green, 0; blue, 0 }  ][line width=0.75]      (0, 0) circle [x radius= 2.01, y radius= 2.01]   ;
\draw    (130,90) .. controls (139.8,102.5) and (169,103.5) .. (170,120) ;
\draw [shift={(170,120)}, rotate = 86.53] [color={rgb, 255:red, 0; green, 0; blue, 0 }  ][fill={rgb, 255:red, 0; green, 0; blue, 0 }  ][line width=0.75]      (0, 0) circle [x radius= 2.01, y radius= 2.01]   ;
\draw    (90,120) .. controls (86.8,132.45) and (82.8,128.45) .. (81.3,138.45) .. controls (81.3,156.45) and (69.5,141.95) .. (65,160) ;
\draw [shift={(65,160)}, rotate = 104] [color={rgb, 255:red, 0; green, 0; blue, 0 }  ][fill={rgb, 255:red, 0; green, 0; blue, 0 }  ][line width=0.75]      (0, 0) circle [x radius= 2.01, y radius= 2.01]   ;
\draw    (90,120) .. controls (107.32,128.8) and (98.32,143.2) .. (110,160) ;
\draw [shift={(110,160)}, rotate = 55.19] [color={rgb, 255:red, 0; green, 0; blue, 0 }  ][fill={rgb, 255:red, 0; green, 0; blue, 0 }  ][line width=0.75]      (0, 0) circle [x radius= 2.01, y radius= 2.01]   ;
\draw    (170,120) .. controls (166.8,132.45) and (156.74,134.96) .. (155.24,144.96) .. controls (155.24,162.96) and (153.2,157.2) .. (150,170) ;
\draw [shift={(150,170)}, rotate = 104.04] [color={rgb, 255:red, 0; green, 0; blue, 0 }  ][fill={rgb, 255:red, 0; green, 0; blue, 0 }  ][line width=0.75]      (0, 0) circle [x radius= 2.01, y radius= 2.01]   ;
\draw    (170,120) .. controls (178.04,130.56) and (183.87,132.97) .. (190.04,141.36) .. controls (198.04,154.56) and (180.4,154) .. (190,170) ;
\draw [shift={(190,170)}, rotate = 59.04] [color={rgb, 255:red, 0; green, 0; blue, 0 }  ][fill={rgb, 255:red, 0; green, 0; blue, 0 }  ][line width=0.75]      (0, 0) circle [x radius= 2.01, y radius= 2.01]   ;
\draw    (65,160) .. controls (68.88,166.73) and (48.81,181.38) .. (50,190) ;
\draw [shift={(50,190)}, rotate = 82.15] [color={rgb, 255:red, 0; green, 0; blue, 0 }  ][fill={rgb, 255:red, 0; green, 0; blue, 0 }  ][line width=0.75]      (0, 0) circle [x radius= 2.01, y radius= 2.01]   ;
\draw    (65,160) .. controls (68.88,166.73) and (79.6,181.11) .. (80,190) ;
\draw [shift={(80,190)}, rotate = 87.45] [color={rgb, 255:red, 0; green, 0; blue, 0 }  ][fill={rgb, 255:red, 0; green, 0; blue, 0 }  ][line width=0.75]      (0, 0) circle [x radius= 2.01, y radius= 2.01]   ;
\draw    (50,190) .. controls (48.73,196.71) and (43.2,200.1) .. (43.53,208.77) .. controls (43.53,218.46) and (41.78,220.28) .. (40,230) ;
\draw [shift={(40,230)}, rotate = 100.39] [color={rgb, 255:red, 0; green, 0; blue, 0 }  ][fill={rgb, 255:red, 0; green, 0; blue, 0 }  ][line width=0.75]      (0, 0) circle [x radius= 2.01, y radius= 2.01]   ;
\draw    (80,190) .. controls (87.94,200.52) and (88.67,228) .. (90,235) ;
\draw [shift={(90,235)}, rotate = 79.22] [color={rgb, 255:red, 0; green, 0; blue, 0 }  ][fill={rgb, 255:red, 0; green, 0; blue, 0 }  ][line width=0.75]      (0, 0) circle [x radius= 2.01, y radius= 2.01]   ;
\draw    (80,190) .. controls (78.73,196.71) and (72.46,199.09) .. (73.2,209.43) .. controls (73.53,219.1) and (71.27,228.1) .. (70,235) ;
\draw [shift={(70,235)}, rotate = 100.42] [color={rgb, 255:red, 0; green, 0; blue, 0 }  ][fill={rgb, 255:red, 0; green, 0; blue, 0 }  ][line width=0.75]      (0, 0) circle [x radius= 2.01, y radius= 2.01]   ;
\draw    (50,190) .. controls (53.19,195.69) and (55.49,196.99) .. (57.94,201.51) .. controls (61.11,208.62) and (56.2,221.38) .. (60,230) ;
\draw [shift={(60,230)}, rotate = 66.19] [color={rgb, 255:red, 0; green, 0; blue, 0 }  ][fill={rgb, 255:red, 0; green, 0; blue, 0 }  ][line width=0.75]      (0, 0) circle [x radius= 2.01, y radius= 2.01]   ;
\draw    (190,170) .. controls (193.88,176.73) and (183.81,181.38) .. (185,190) ;
\draw [shift={(185,190)}, rotate = 82.15] [color={rgb, 255:red, 0; green, 0; blue, 0 }  ][fill={rgb, 255:red, 0; green, 0; blue, 0 }  ][line width=0.75]      (0, 0) circle [x radius= 2.01, y radius= 2.01]   ;
\draw    (190,170) .. controls (193.88,176.73) and (204.6,186.11) .. (205,195) ;
\draw [shift={(205,195)}, rotate = 87.45] [color={rgb, 255:red, 0; green, 0; blue, 0 }  ][fill={rgb, 255:red, 0; green, 0; blue, 0 }  ][line width=0.75]      (0, 0) circle [x radius= 2.01, y radius= 2.01]   ;
\draw    (185,190) .. controls (183.73,196.71) and (178.2,200.1) .. (178.53,208.77) .. controls (178.53,218.46) and (181.78,220.28) .. (180,230) ;
\draw [shift={(180,230)}, rotate = 100.39] [color={rgb, 255:red, 0; green, 0; blue, 0 }  ][fill={rgb, 255:red, 0; green, 0; blue, 0 }  ][line width=0.75]      (0, 0) circle [x radius= 2.01, y radius= 2.01]   ;
\draw    (205,195) .. controls (217.32,202.4) and (207.2,214.8) .. (220,230) ;
\draw [shift={(220,230)}, rotate = 49.9] [color={rgb, 255:red, 0; green, 0; blue, 0 }  ][fill={rgb, 255:red, 0; green, 0; blue, 0 }  ][line width=0.75]      (0, 0) circle [x radius= 2.01, y radius= 2.01]   ;
\draw    (205,195) .. controls (203.73,201.71) and (197.46,204.09) .. (198.2,214.43) .. controls (198.53,224.1) and (201.27,223.1) .. (200,230) ;
\draw [shift={(200,230)}, rotate = 100.42] [color={rgb, 255:red, 0; green, 0; blue, 0 }  ][fill={rgb, 255:red, 0; green, 0; blue, 0 }  ][line width=0.75]      (0, 0) circle [x radius= 2.01, y radius= 2.01]   ;
\draw    (185,190) .. controls (188.19,195.69) and (190.49,196.99) .. (192.94,201.51) .. controls (196.11,208.62) and (186.2,221.38) .. (190,230) ;
\draw [shift={(190,230)}, rotate = 66.19] [color={rgb, 255:red, 0; green, 0; blue, 0 }  ][fill={rgb, 255:red, 0; green, 0; blue, 0 }  ][line width=0.75]      (0, 0) circle [x radius= 2.01, y radius= 2.01]   ;
\draw    (150,170) .. controls (135.32,178.32) and (150.12,184.32) .. (140,200) ;
\draw [shift={(140,200)}, rotate = 122.84] [color={rgb, 255:red, 0; green, 0; blue, 0 }  ][fill={rgb, 255:red, 0; green, 0; blue, 0 }  ][line width=0.75]      (0, 0) circle [x radius= 2.01, y radius= 2.01]   ;
\draw    (150,170) .. controls (153.88,176.73) and (164.6,191.11) .. (165,200) ;
\draw [shift={(165,200)}, rotate = 87.45] [color={rgb, 255:red, 0; green, 0; blue, 0 }  ][fill={rgb, 255:red, 0; green, 0; blue, 0 }  ][line width=0.75]      (0, 0) circle [x radius= 2.01, y radius= 2.01]   ;
\draw    (140,200) .. controls (138.73,206.71) and (133.7,215.25) .. (134.2,225.25) .. controls (135.2,233.75) and (131.78,240.28) .. (130,250) ;
\draw [shift={(130,250)}, rotate = 100.39] [color={rgb, 255:red, 0; green, 0; blue, 0 }  ][fill={rgb, 255:red, 0; green, 0; blue, 0 }  ][line width=0.75]      (0, 0) circle [x radius= 2.01, y radius= 2.01]   ;
\draw    (165,200) .. controls (169.72,212.8) and (170.92,234.4) .. (175,240) ;
\draw [shift={(175,240)}, rotate = 53.92] [color={rgb, 255:red, 0; green, 0; blue, 0 }  ][fill={rgb, 255:red, 0; green, 0; blue, 0 }  ][line width=0.75]      (0, 0) circle [x radius= 2.01, y radius= 2.01]   ;
\draw    (165,200) .. controls (163.73,206.71) and (153.72,210.8) .. (159.32,220) .. controls (163.72,228) and (148.12,230.8) .. (155,240) ;
\draw [shift={(155,240)}, rotate = 53.21] [color={rgb, 255:red, 0; green, 0; blue, 0 }  ][fill={rgb, 255:red, 0; green, 0; blue, 0 }  ][line width=0.75]      (0, 0) circle [x radius= 2.01, y radius= 2.01]   ;
\draw    (140,200) .. controls (143.19,205.69) and (147.56,223.2) .. (150,250) ;
\draw [shift={(150,250)}, rotate = 84.8] [color={rgb, 255:red, 0; green, 0; blue, 0 }  ][fill={rgb, 255:red, 0; green, 0; blue, 0 }  ][line width=0.75]      (0, 0) circle [x radius= 2.01, y radius= 2.01]   ;
\draw    (110,160) .. controls (110.12,170) and (96.4,172) .. (100,185) ;
\draw [shift={(100,185)}, rotate = 74.52] [color={rgb, 255:red, 0; green, 0; blue, 0 }  ][fill={rgb, 255:red, 0; green, 0; blue, 0 }  ][line width=0.75]      (0, 0) circle [x radius= 2.01, y radius= 2.01]   ;
\draw    (110,160) .. controls (113.88,166.73) and (119.6,176.11) .. (120,185) ;
\draw [shift={(120,185)}, rotate = 87.45] [color={rgb, 255:red, 0; green, 0; blue, 0 }  ][fill={rgb, 255:red, 0; green, 0; blue, 0 }  ][line width=0.75]      (0, 0) circle [x radius= 2.01, y radius= 2.01]   ;
\draw    (100,185) .. controls (98.73,191.71) and (96.19,193.33) .. (96.52,202) .. controls (96.52,211.7) and (98.3,210.28) .. (96.52,220) ;
\draw [shift={(96.52,220)}, rotate = 100.39] [color={rgb, 255:red, 0; green, 0; blue, 0 }  ][fill={rgb, 255:red, 0; green, 0; blue, 0 }  ][line width=0.75]      (0, 0) circle [x radius= 2.01, y radius= 2.01]   ;
\draw    (120,185) .. controls (126.92,195.36) and (126.4,217.2) .. (125,230) ;
\draw [shift={(125,230)}, rotate = 96.24] [color={rgb, 255:red, 0; green, 0; blue, 0 }  ][fill={rgb, 255:red, 0; green, 0; blue, 0 }  ][line width=0.75]      (0, 0) circle [x radius= 2.01, y radius= 2.01]   ;
\draw    (120,185) .. controls (118.73,191.71) and (114.26,219.65) .. (115,230) ;
\draw [shift={(115,230)}, rotate = 85.89] [color={rgb, 255:red, 0; green, 0; blue, 0 }  ][fill={rgb, 255:red, 0; green, 0; blue, 0 }  ][line width=0.75]      (0, 0) circle [x radius= 2.01, y radius= 2.01]   ;
\draw    (100,185) .. controls (103.19,190.69) and (105.9,195.09) .. (106.12,200.4) .. controls (107.5,209.09) and (101.2,211.38) .. (105,220) ;
\draw [shift={(105,220)}, rotate = 66.19] [color={rgb, 255:red, 0; green, 0; blue, 0 }  ][fill={rgb, 255:red, 0; green, 0; blue, 0 }  ][line width=0.75]      (0, 0) circle [x radius= 2.01, y radius= 2.01]   ;

\end{tikzpicture}

%% file: Tinf-cutoff.tex
\begin{tikzpicture}[x=0.75pt,y=0.75pt,yscale=-.75,xscale=.75]

\draw    (274.48,14.87) -- (274.48,69.43) ;
\draw [shift={(274.48,69.43)}, rotate = 90] [color={rgb, 255:red, 0; green, 0; blue, 0 }  ][fill={rgb, 255:red, 0; green, 0; blue, 0 }  ][line width=0.75]      (0, 0) circle [x radius= 2.01, y radius= 2.01]   ;
\draw [shift={(274.48,42.15)}, rotate = 90] [color={rgb, 255:red, 0; green, 0; blue, 0 }  ][fill={rgb, 255:red, 0; green, 0; blue, 0 }  ][line width=0.75]      (0, 0) circle [x radius= 2.01, y radius= 2.01]   ;
\draw [shift={(274.48,14.87)}, rotate = 90] [color={rgb, 255:red, 0; green, 0; blue, 0 }  ][fill={rgb, 255:red, 0; green, 0; blue, 0 }  ][line width=0.75]      (0, 0) circle [x radius= 2.01, y radius= 2.01]   ;
\draw   (294.93,42.16) -- (305.15,69.43) -- (284.7,69.43) -- cycle ;
\draw    (274.61,15.17) -- (294.93,42.16) ;
\draw    (274.48,69.43) -- (274.48,123.99) ;
\draw [shift={(274.48,123.99)}, rotate = 90] [color={rgb, 255:red, 0; green, 0; blue, 0 }  ][fill={rgb, 255:red, 0; green, 0; blue, 0 }  ][line width=0.75]      (0, 0) circle [x radius= 2.01, y radius= 2.01]   ;
\draw [shift={(274.48,96.71)}, rotate = 90] [color={rgb, 255:red, 0; green, 0; blue, 0 }  ][fill={rgb, 255:red, 0; green, 0; blue, 0 }  ][line width=0.75]      (0, 0) circle [x radius= 2.01, y radius= 2.01]   ;
\draw [shift={(274.48,69.43)}, rotate = 90] [color={rgb, 255:red, 0; green, 0; blue, 0 }  ][fill={rgb, 255:red, 0; green, 0; blue, 0 }  ][line width=0.75]      (0, 0) circle [x radius= 2.01, y radius= 2.01]   ;
\draw    (274.48,123.99) -- (274.48,178.56) ;
\draw [shift={(274.48,178.56)}, rotate = 90] [color={rgb, 255:red, 0; green, 0; blue, 0 }  ][fill={rgb, 255:red, 0; green, 0; blue, 0 }  ][line width=0.75]      (0, 0) circle [x radius= 2.01, y radius= 2.01]   ;
\draw [shift={(274.48,151.28)}, rotate = 90] [color={rgb, 255:red, 0; green, 0; blue, 0 }  ][fill={rgb, 255:red, 0; green, 0; blue, 0 }  ][line width=0.75]      (0, 0) circle [x radius= 2.01, y radius= 2.01]   ;
\draw [shift={(274.48,123.99)}, rotate = 90] [color={rgb, 255:red, 0; green, 0; blue, 0 }  ][fill={rgb, 255:red, 0; green, 0; blue, 0 }  ][line width=0.75]      (0, 0) circle [x radius= 2.01, y radius= 2.01]   ;
\draw   (254.02,42.16) -- (264.25,69.43) -- (243.8,69.43) -- cycle ;
\draw    (274.48,14.87) -- (254.02,42.16) ;
\draw   (294.79,96.43) -- (305.02,123.7) -- (284.57,123.7) -- cycle ;
\draw    (274.48,69.43) -- (294.79,96.43) ;
\draw   (254.02,124.01) -- (264.25,151.28) -- (243.8,151.28) -- cycle ;
\draw    (274.48,96.71) -- (254.02,124.01) ;
\draw   (230.16,42.14) -- (240.39,69.41) -- (219.94,69.41) -- cycle ;
\draw    (274.61,15.17) -- (230.16,42.14) ;
\draw   (318.79,42.16) -- (329.02,69.43) -- (308.56,69.43) -- cycle ;
\draw    (274.48,14.87) -- (318.79,42.16) ;
\draw   (318.79,96.73) -- (329.02,123.99) -- (308.56,123.99) -- cycle ;
\draw    (274.48,69.43) -- (318.79,96.73) ;
\draw   (230.03,123.69) -- (240.25,150.96) -- (219.8,150.96) -- cycle ;
\draw    (274.48,96.71) -- (230.03,123.69) ;
\draw   (294.79,205.55) -- (305.02,232.82) -- (284.57,232.82) -- cycle ;
\draw    (274.48,178.56) -- (294.79,205.55) ;
\draw   (318.79,205.85) -- (329.02,233.12) -- (308.56,233.12) -- cycle ;
\draw    (274.48,178.56) -- (318.79,205.85) ;
\draw   (254.02,205.85) -- (264.25,233.12) -- (243.8,233.12) -- cycle ;
\draw    (274.48,178.56) -- (254.02,205.85) ;
\draw   (294.79,150.99) -- (305.02,178.26) -- (284.57,178.26) -- cycle ;
\draw    (274.48,123.99) -- (294.79,150.99) ;
\draw    (342.83,14.87) -- (342.83,178.6) ;
\draw [shift={(342.83,178.6)}, rotate = 270] [color={rgb, 255:red, 0; green, 0; blue, 0 }  ][line width=0.75]    (0,5.59) -- (0,-5.59)   ;
\draw [shift={(342.83,14.87)}, rotate = 270] [color={rgb, 255:red, 0; green, 0; blue, 0 }  ][line width=0.75]    (0,5.59) -- (0,-5.59)   ;

\draw (297, 15.5) node   [align=left] {\begin{minipage}[lt]{8.67pt}\setlength\topsep{0pt}
$\displaystyle \rho $
\end{minipage}};
\draw (356, 94.85) node   [align=left] {\begin{minipage}[lt]{8.67pt}\setlength\topsep{0pt}
$\displaystyle \ell $
\end{minipage}};

\end{tikzpicture}

%% file: HS_french.tex
\begin{tikzpicture}[x=0.75pt,y=0.75pt,yscale=-.75,xscale=.75]

\draw    (160,20) -- (160,60) ;
\draw [shift={(160,60)}, rotate = 90] [color={rgb, 255:red, 0; green, 0; blue, 0 }  ][fill={rgb, 255:red, 0; green, 0; blue, 0 }  ][line width=0.75]      (0, 0) circle [x radius= 2.01, y radius= 2.01]   ;
\draw [shift={(160,20)}, rotate = 90] [color={rgb, 255:red, 0; green, 0; blue, 0 }  ][fill={rgb, 255:red, 0; green, 0; blue, 0 }  ][line width=0.75]      (0, 0) circle [x radius= 2.01, y radius= 2.01]   ;
\draw    (160,20) -- (255,60) ;
\draw [shift={(255,60)}, rotate = 22.83] [color={rgb, 255:red, 0; green, 0; blue, 0 }  ][fill={rgb, 255:red, 0; green, 0; blue, 0 }  ][line width=0.75]      (0, 0) circle [x radius= 2.01, y radius= 2.01]   ;
\draw    (160,20) -- (65,60.06) ;
\draw [shift={(65,60.06)}, rotate = 157.14] [color={rgb, 255:red, 0; green, 0; blue, 0 }  ][fill={rgb, 255:red, 0; green, 0; blue, 0 }  ][line width=0.75]      (0, 0) circle [x radius= 2.01, y radius= 2.01]   ;
\draw    (160,60) -- (185,100) ;
\draw [shift={(185,100)}, rotate = 57.99] [color={rgb, 255:red, 0; green, 0; blue, 0 }  ][fill={rgb, 255:red, 0; green, 0; blue, 0 }  ][line width=0.75]      (0, 0) circle [x radius= 2.01, y radius= 2.01]   ;
\draw    (160,60) -- (135,100) ;
\draw [shift={(135,100)}, rotate = 122.01] [color={rgb, 255:red, 0; green, 0; blue, 0 }  ][fill={rgb, 255:red, 0; green, 0; blue, 0 }  ][line width=0.75]      (0, 0) circle [x radius= 2.01, y radius= 2.01]   ;
\draw    (255,60) -- (280,100) ;
\draw [shift={(280,100)}, rotate = 57.99] [color={rgb, 255:red, 0; green, 0; blue, 0 }  ][fill={rgb, 255:red, 0; green, 0; blue, 0 }  ][line width=0.75]      (0, 0) circle [x radius= 2.01, y radius= 2.01]   ;
\draw    (255,60) -- (230,100) ;
\draw [shift={(230,100)}, rotate = 122.01] [color={rgb, 255:red, 0; green, 0; blue, 0 }  ][fill={rgb, 255:red, 0; green, 0; blue, 0 }  ][line width=0.75]      (0, 0) circle [x radius= 2.01, y radius= 2.01]   ;
\draw    (230,100) -- (230,140) ;
\draw [shift={(230,140)}, rotate = 90] [color={rgb, 255:red, 0; green, 0; blue, 0 }  ][fill={rgb, 255:red, 0; green, 0; blue, 0 }  ][line width=0.75]      (0, 0) circle [x radius= 2.01, y radius= 2.01]   ;
\draw    (185,100) -- (200,140) ;
\draw [shift={(200,140)}, rotate = 69.44] [color={rgb, 255:red, 0; green, 0; blue, 0 }  ][fill={rgb, 255:red, 0; green, 0; blue, 0 }  ][line width=0.75]      (0, 0) circle [x radius= 2.01, y radius= 2.01]   ;
\draw    (185,100) -- (170,140) ;
\draw [shift={(170,140)}, rotate = 110.56] [color={rgb, 255:red, 0; green, 0; blue, 0 }  ][fill={rgb, 255:red, 0; green, 0; blue, 0 }  ][line width=0.75]      (0, 0) circle [x radius= 2.01, y radius= 2.01]   ;
\draw    (185,100) -- (185,140) ;
\draw [shift={(185,140)}, rotate = 90] [color={rgb, 255:red, 0; green, 0; blue, 0 }  ][fill={rgb, 255:red, 0; green, 0; blue, 0 }  ][line width=0.75]      (0, 0) circle [x radius= 2.01, y radius= 2.01]   ;
\draw    (135,100) -- (150,140) ;
\draw [shift={(150,140)}, rotate = 69.44] [color={rgb, 255:red, 0; green, 0; blue, 0 }  ][fill={rgb, 255:red, 0; green, 0; blue, 0 }  ][line width=0.75]      (0, 0) circle [x radius= 2.01, y radius= 2.01]   ;
\draw    (135,100) -- (120,140) ;
\draw [shift={(120,140)}, rotate = 110.56] [color={rgb, 255:red, 0; green, 0; blue, 0 }  ][fill={rgb, 255:red, 0; green, 0; blue, 0 }  ][line width=0.75]      (0, 0) circle [x radius= 2.01, y radius= 2.01]   ;
\draw    (135,100) -- (135,140) ;
\draw [shift={(135,140)}, rotate = 90] [color={rgb, 255:red, 0; green, 0; blue, 0 }  ][fill={rgb, 255:red, 0; green, 0; blue, 0 }  ][line width=0.75]      (0, 0) circle [x radius= 2.01, y radius= 2.01]   ;
\draw    (65,60.06) -- (90,100) ;
\draw [shift={(90,100)}, rotate = 57.96] [color={rgb, 255:red, 0; green, 0; blue, 0 }  ][fill={rgb, 255:red, 0; green, 0; blue, 0 }  ][line width=0.75]      (0, 0) circle [x radius= 2.01, y radius= 2.01]   ;
\draw    (65,60.06) -- (40,100) ;
\draw [shift={(40,100)}, rotate = 122.04] [color={rgb, 255:red, 0; green, 0; blue, 0 }  ][fill={rgb, 255:red, 0; green, 0; blue, 0 }  ][line width=0.75]      (0, 0) circle [x radius= 2.01, y radius= 2.01]   ;
\draw    (90,100) -- (90,140) ;
\draw [shift={(90,140)}, rotate = 90] [color={rgb, 255:red, 0; green, 0; blue, 0 }  ][fill={rgb, 255:red, 0; green, 0; blue, 0 }  ][line width=0.75]      (0, 0) circle [x radius= 2.01, y radius= 2.01]   ;
\draw    (40,100) -- (25,140) ;
\draw [shift={(25,140)}, rotate = 110.56] [color={rgb, 255:red, 0; green, 0; blue, 0 }  ][fill={rgb, 255:red, 0; green, 0; blue, 0 }  ][line width=0.75]      (0, 0) circle [x radius= 2.01, y radius= 2.01]   ;
\draw    (230,140) -- (245,180) ;
\draw [shift={(245,180)}, rotate = 69.44] [color={rgb, 255:red, 0; green, 0; blue, 0 }  ][fill={rgb, 255:red, 0; green, 0; blue, 0 }  ][line width=0.75]      (0, 0) circle [x radius= 2.01, y radius= 2.01]   ;
\draw    (230,140) -- (215,180) ;
\draw [shift={(215,180)}, rotate = 110.56] [color={rgb, 255:red, 0; green, 0; blue, 0 }  ][fill={rgb, 255:red, 0; green, 0; blue, 0 }  ][line width=0.75]      (0, 0) circle [x radius= 2.01, y radius= 2.01]   ;
\draw    (230,140) -- (230,180) ;
\draw [shift={(230,180)}, rotate = 90] [color={rgb, 255:red, 0; green, 0; blue, 0 }  ][fill={rgb, 255:red, 0; green, 0; blue, 0 }  ][line width=0.75]      (0, 0) circle [x radius= 2.01, y radius= 2.01]   ;
\draw    (90,140) -- (105,180) ;
\draw [shift={(105,180)}, rotate = 69.44] [color={rgb, 255:red, 0; green, 0; blue, 0 }  ][fill={rgb, 255:red, 0; green, 0; blue, 0 }  ][line width=0.75]      (0, 0) circle [x radius= 2.01, y radius= 2.01]   ;
\draw    (90,140) -- (75,180) ;
\draw [shift={(75,180)}, rotate = 110.56] [color={rgb, 255:red, 0; green, 0; blue, 0 }  ][fill={rgb, 255:red, 0; green, 0; blue, 0 }  ][line width=0.75]      (0, 0) circle [x radius= 2.01, y radius= 2.01]   ;
\draw    (90,140) -- (90,180) ;
\draw [shift={(90,180)}, rotate = 90] [color={rgb, 255:red, 0; green, 0; blue, 0 }  ][fill={rgb, 255:red, 0; green, 0; blue, 0 }  ][line width=0.75]      (0, 0) circle [x radius= 2.01, y radius= 2.01]   ;
\draw    (25,140) -- (40,180) ;
\draw [shift={(40,180)}, rotate = 69.44] [color={rgb, 255:red, 0; green, 0; blue, 0 }  ][fill={rgb, 255:red, 0; green, 0; blue, 0 }  ][line width=0.75]      (0, 0) circle [x radius= 2.01, y radius= 2.01]   ;
\draw    (25,140) -- (10,180) ;
\draw [shift={(10,180)}, rotate = 110.56] [color={rgb, 255:red, 0; green, 0; blue, 0 }  ][fill={rgb, 255:red, 0; green, 0; blue, 0 }  ][line width=0.75]      (0, 0) circle [x radius= 2.01, y radius= 2.01]   ;
\draw    (40,100) -- (55,140) ;
\draw [shift={(55,140)}, rotate = 69.44] [color={rgb, 255:red, 0; green, 0; blue, 0 }  ][fill={rgb, 255:red, 0; green, 0; blue, 0 }  ][line width=0.75]      (0, 0) circle [x radius= 2.01, y radius= 2.01]   ;

\draw (175.29,17.85) node  [font=\small,color={rgb, 255:red, 74; green, 144; blue, 226 }  ,opacity=1 ] [align=left] {\begin{minipage}[lt]{8.67pt}\setlength\topsep{0pt}
$\displaystyle 4$
\end{minipage}};
\draw (175.29,57.85) node  [font=\small,color={rgb, 255:red, 74; green, 144; blue, 226 }  ,opacity=1 ] [align=left] {\begin{minipage}[lt]{8.67pt}\setlength\topsep{0pt}
$\displaystyle 3$
\end{minipage}};
\draw (201.29,97.85) node  [font=\small,color={rgb, 255:red, 74; green, 144; blue, 226 }  ,opacity=1 ] [align=left] {\begin{minipage}[lt]{8.67pt}\setlength\topsep{0pt}
$\displaystyle 2$
\end{minipage}};
\draw (150.29,97.85) node  [font=\small,color={rgb, 255:red, 74; green, 144; blue, 226 }  ,opacity=1 ] [align=left] {\begin{minipage}[lt]{8.67pt}\setlength\topsep{0pt}
$\displaystyle 2$
\end{minipage}};
\draw (245.29,97.85) node  [font=\small,color={rgb, 255:red, 74; green, 144; blue, 226 }  ,opacity=1 ] [align=left] {\begin{minipage}[lt]{8.67pt}\setlength\topsep{0pt}
$\displaystyle 2$
\end{minipage}};
\draw (105.29,97.85) node  [font=\small,color={rgb, 255:red, 74; green, 144; blue, 226 }  ,opacity=1 ] [align=left] {\begin{minipage}[lt]{8.67pt}\setlength\topsep{0pt}
$\displaystyle 2$
\end{minipage}};
\draw (55.29,97.85) node  [font=\small,color={rgb, 255:red, 74; green, 144; blue, 226 }  ,opacity=1 ] [align=left] {\begin{minipage}[lt]{8.67pt}\setlength\topsep{0pt}
$\displaystyle 1$
\end{minipage}};
\draw (40.29,137.85) node  [font=\small,color={rgb, 255:red, 74; green, 144; blue, 226 }  ,opacity=1 ] [align=left] {\begin{minipage}[lt]{8.67pt}\setlength\topsep{0pt}
$\displaystyle 1$
\end{minipage}};
\draw (105.29,137.85) node  [font=\small,color={rgb, 255:red, 74; green, 144; blue, 226 }  ,opacity=1 ] [align=left] {\begin{minipage}[lt]{8.67pt}\setlength\topsep{0pt}
$\displaystyle 2$
\end{minipage}};
\draw (245.29,137.85) node  [font=\small,color={rgb, 255:red, 74; green, 144; blue, 226 }  ,opacity=1 ] [align=left] {\begin{minipage}[lt]{8.67pt}\setlength\topsep{0pt}
$\displaystyle 2$
\end{minipage}};
\draw (270.29,57.85) node  [font=\small,color={rgb, 255:red, 74; green, 144; blue, 226 }  ,opacity=1 ] [align=left] {\begin{minipage}[lt]{8.67pt}\setlength\topsep{0pt}
$\displaystyle 2$
\end{minipage}};
\draw (55.29,57.85) node  [font=\small,color={rgb, 255:red, 74; green, 144; blue, 226 }  ,opacity=1 ] [align=left] {\begin{minipage}[lt]{8.67pt}\setlength\topsep{0pt}
$\displaystyle 2$
\end{minipage}};

\end{tikzpicture}

%% file: HS_canadian.tex
\begin{tikzpicture}[x=0.75pt,y=0.75pt,yscale=-.75,xscale=.75]

\draw    (160,20) -- (160,60) ;
\draw [shift={(160,60)}, rotate = 90] [color={rgb, 255:red, 0; green, 0; blue, 0 }  ][fill={rgb, 255:red, 0; green, 0; blue, 0 }  ][line width=0.75]      (0, 0) circle [x radius= 2.01, y radius= 2.01]   ;
\draw [shift={(160,20)}, rotate = 90] [color={rgb, 255:red, 0; green, 0; blue, 0 }  ][fill={rgb, 255:red, 0; green, 0; blue, 0 }  ][line width=0.75]      (0, 0) circle [x radius= 2.01, y radius= 2.01]   ;
\draw    (160,20) -- (255,60) ;
\draw [shift={(255,60)}, rotate = 22.83] [color={rgb, 255:red, 0; green, 0; blue, 0 }  ][fill={rgb, 255:red, 0; green, 0; blue, 0 }  ][line width=0.75]      (0, 0) circle [x radius= 2.01, y radius= 2.01]   ;
\draw    (160,20) -- (65,60.06) ;
\draw [shift={(65,60.06)}, rotate = 157.14] [color={rgb, 255:red, 0; green, 0; blue, 0 }  ][fill={rgb, 255:red, 0; green, 0; blue, 0 }  ][line width=0.75]      (0, 0) circle [x radius= 2.01, y radius= 2.01]   ;
\draw    (160,60) -- (185,100) ;
\draw [shift={(185,100)}, rotate = 57.99] [color={rgb, 255:red, 0; green, 0; blue, 0 }  ][fill={rgb, 255:red, 0; green, 0; blue, 0 }  ][line width=0.75]      (0, 0) circle [x radius= 2.01, y radius= 2.01]   ;
\draw    (160,60) -- (135,100) ;
\draw [shift={(135,100)}, rotate = 122.01] [color={rgb, 255:red, 0; green, 0; blue, 0 }  ][fill={rgb, 255:red, 0; green, 0; blue, 0 }  ][line width=0.75]      (0, 0) circle [x radius= 2.01, y radius= 2.01]   ;
\draw    (255,60) -- (280,100) ;
\draw [shift={(280,100)}, rotate = 57.99] [color={rgb, 255:red, 0; green, 0; blue, 0 }  ][fill={rgb, 255:red, 0; green, 0; blue, 0 }  ][line width=0.75]      (0, 0) circle [x radius= 2.01, y radius= 2.01]   ;
\draw    (255,60) -- (230,100) ;
\draw [shift={(230,100)}, rotate = 122.01] [color={rgb, 255:red, 0; green, 0; blue, 0 }  ][fill={rgb, 255:red, 0; green, 0; blue, 0 }  ][line width=0.75]      (0, 0) circle [x radius= 2.01, y radius= 2.01]   ;
\draw    (230,100) -- (230,140) ;
\draw [shift={(230,140)}, rotate = 90] [color={rgb, 255:red, 0; green, 0; blue, 0 }  ][fill={rgb, 255:red, 0; green, 0; blue, 0 }  ][line width=0.75]      (0, 0) circle [x radius= 2.01, y radius= 2.01]   ;
\draw    (185,100) -- (200,140) ;
\draw [shift={(200,140)}, rotate = 69.44] [color={rgb, 255:red, 0; green, 0; blue, 0 }  ][fill={rgb, 255:red, 0; green, 0; blue, 0 }  ][line width=0.75]      (0, 0) circle [x radius= 2.01, y radius= 2.01]   ;
\draw    (185,100) -- (170,140) ;
\draw [shift={(170,140)}, rotate = 110.56] [color={rgb, 255:red, 0; green, 0; blue, 0 }  ][fill={rgb, 255:red, 0; green, 0; blue, 0 }  ][line width=0.75]      (0, 0) circle [x radius= 2.01, y radius= 2.01]   ;
\draw    (185,100) -- (185,140) ;
\draw [shift={(185,140)}, rotate = 90] [color={rgb, 255:red, 0; green, 0; blue, 0 }  ][fill={rgb, 255:red, 0; green, 0; blue, 0 }  ][line width=0.75]      (0, 0) circle [x radius= 2.01, y radius= 2.01]   ;
\draw    (135,100) -- (150,140) ;
\draw [shift={(150,140)}, rotate = 69.44] [color={rgb, 255:red, 0; green, 0; blue, 0 }  ][fill={rgb, 255:red, 0; green, 0; blue, 0 }  ][line width=0.75]      (0, 0) circle [x radius= 2.01, y radius= 2.01]   ;
\draw    (135,100) -- (120,140) ;
\draw [shift={(120,140)}, rotate = 110.56] [color={rgb, 255:red, 0; green, 0; blue, 0 }  ][fill={rgb, 255:red, 0; green, 0; blue, 0 }  ][line width=0.75]      (0, 0) circle [x radius= 2.01, y radius= 2.01]   ;
\draw    (135,100) -- (135,140) ;
\draw [shift={(135,140)}, rotate = 90] [color={rgb, 255:red, 0; green, 0; blue, 0 }  ][fill={rgb, 255:red, 0; green, 0; blue, 0 }  ][line width=0.75]      (0, 0) circle [x radius= 2.01, y radius= 2.01]   ;
\draw    (65,60.06) -- (90,100) ;
\draw [shift={(90,100)}, rotate = 57.96] [color={rgb, 255:red, 0; green, 0; blue, 0 }  ][fill={rgb, 255:red, 0; green, 0; blue, 0 }  ][line width=0.75]      (0, 0) circle [x radius= 2.01, y radius= 2.01]   ;
\draw    (65,60.06) -- (40,100) ;
\draw [shift={(40,100)}, rotate = 122.04] [color={rgb, 255:red, 0; green, 0; blue, 0 }  ][fill={rgb, 255:red, 0; green, 0; blue, 0 }  ][line width=0.75]      (0, 0) circle [x radius= 2.01, y radius= 2.01]   ;
\draw    (90,100) -- (90,140) ;
\draw [shift={(90,140)}, rotate = 90] [color={rgb, 255:red, 0; green, 0; blue, 0 }  ][fill={rgb, 255:red, 0; green, 0; blue, 0 }  ][line width=0.75]      (0, 0) circle [x radius= 2.01, y radius= 2.01]   ;
\draw    (40,100) -- (25,140) ;
\draw [shift={(25,140)}, rotate = 110.56] [color={rgb, 255:red, 0; green, 0; blue, 0 }  ][fill={rgb, 255:red, 0; green, 0; blue, 0 }  ][line width=0.75]      (0, 0) circle [x radius= 2.01, y radius= 2.01]   ;
\draw    (230,140) -- (245,180) ;
\draw [shift={(245,180)}, rotate = 69.44] [color={rgb, 255:red, 0; green, 0; blue, 0 }  ][fill={rgb, 255:red, 0; green, 0; blue, 0 }  ][line width=0.75]      (0, 0) circle [x radius= 2.01, y radius= 2.01]   ;
\draw    (230,140) -- (215,180) ;
\draw [shift={(215,180)}, rotate = 110.56] [color={rgb, 255:red, 0; green, 0; blue, 0 }  ][fill={rgb, 255:red, 0; green, 0; blue, 0 }  ][line width=0.75]      (0, 0) circle [x radius= 2.01, y radius= 2.01]   ;
\draw    (230,140) -- (230,180) ;
\draw [shift={(230,180)}, rotate = 90] [color={rgb, 255:red, 0; green, 0; blue, 0 }  ][fill={rgb, 255:red, 0; green, 0; blue, 0 }  ][line width=0.75]      (0, 0) circle [x radius= 2.01, y radius= 2.01]   ;
\draw    (90,140) -- (105,180) ;
\draw [shift={(105,180)}, rotate = 69.44] [color={rgb, 255:red, 0; green, 0; blue, 0 }  ][fill={rgb, 255:red, 0; green, 0; blue, 0 }  ][line width=0.75]      (0, 0) circle [x radius= 2.01, y radius= 2.01]   ;
\draw    (90,140) -- (75,180) ;
\draw [shift={(75,180)}, rotate = 110.56] [color={rgb, 255:red, 0; green, 0; blue, 0 }  ][fill={rgb, 255:red, 0; green, 0; blue, 0 }  ][line width=0.75]      (0, 0) circle [x radius= 2.01, y radius= 2.01]   ;
\draw    (90,140) -- (90,180) ;
\draw [shift={(90,180)}, rotate = 90] [color={rgb, 255:red, 0; green, 0; blue, 0 }  ][fill={rgb, 255:red, 0; green, 0; blue, 0 }  ][line width=0.75]      (0, 0) circle [x radius= 2.01, y radius= 2.01]   ;
\draw    (25,140) -- (40,180) ;
\draw [shift={(40,180)}, rotate = 69.44] [color={rgb, 255:red, 0; green, 0; blue, 0 }  ][fill={rgb, 255:red, 0; green, 0; blue, 0 }  ][line width=0.75]      (0, 0) circle [x radius= 2.01, y radius= 2.01]   ;
\draw    (25,140) -- (10,180) ;
\draw [shift={(10,180)}, rotate = 110.56] [color={rgb, 255:red, 0; green, 0; blue, 0 }  ][fill={rgb, 255:red, 0; green, 0; blue, 0 }  ][line width=0.75]      (0, 0) circle [x radius= 2.01, y radius= 2.01]   ;
\draw    (40,100) -- (55,140) ;
\draw [shift={(55,140)}, rotate = 69.44] [color={rgb, 255:red, 0; green, 0; blue, 0 }  ][fill={rgb, 255:red, 0; green, 0; blue, 0 }  ][line width=0.75]      (0, 0) circle [x radius= 2.01, y radius= 2.01]   ;

\draw (175.29,17.85) node  [font=\small,color={rgb, 255:red, 74; green, 144; blue, 226 }  ,opacity=1 ] [align=left] {\begin{minipage}[lt]{8.67pt}\setlength\topsep{0pt}
$\displaystyle 3$
\end{minipage}};
\draw (175.29,57.85) node  [font=\small,color={rgb, 255:red, 74; green, 144; blue, 226 }  ,opacity=1 ] [align=left] {\begin{minipage}[lt]{8.67pt}\setlength\topsep{0pt}
$\displaystyle 3$
\end{minipage}};
\draw (201.29,97.85) node  [font=\small,color={rgb, 255:red, 74; green, 144; blue, 226 }  ,opacity=1 ] [align=left] {\begin{minipage}[lt]{8.67pt}\setlength\topsep{0pt}
$\displaystyle 2$
\end{minipage}};
\draw (150.29,97.85) node  [font=\small,color={rgb, 255:red, 74; green, 144; blue, 226 }  ,opacity=1 ] [align=left] {\begin{minipage}[lt]{8.67pt}\setlength\topsep{0pt}
$\displaystyle 2$
\end{minipage}};
\draw (245.29,97.85) node  [font=\small,color={rgb, 255:red, 74; green, 144; blue, 226 }  ,opacity=1 ] [align=left] {\begin{minipage}[lt]{8.67pt}\setlength\topsep{0pt}
$\displaystyle 2$
\end{minipage}};
\draw (105.29,97.85) node  [font=\small,color={rgb, 255:red, 74; green, 144; blue, 226 }  ,opacity=1 ] [align=left] {\begin{minipage}[lt]{8.67pt}\setlength\topsep{0pt}
$\displaystyle 2$
\end{minipage}};
\draw (55.29,97.85) node  [font=\small,color={rgb, 255:red, 74; green, 144; blue, 226 }  ,opacity=1 ] [align=left] {\begin{minipage}[lt]{8.67pt}\setlength\topsep{0pt}
$\displaystyle 1$
\end{minipage}};
\draw (40.29,137.85) node  [font=\small,color={rgb, 255:red, 74; green, 144; blue, 226 }  ,opacity=1 ] [align=left] {\begin{minipage}[lt]{8.67pt}\setlength\topsep{0pt}
$\displaystyle 1$
\end{minipage}};
\draw (105.29,137.85) node  [font=\small,color={rgb, 255:red, 74; green, 144; blue, 226 }  ,opacity=1 ] [align=left] {\begin{minipage}[lt]{8.67pt}\setlength\topsep{0pt}
$\displaystyle 2$
\end{minipage}};
\draw (245.29,137.85) node  [font=\small,color={rgb, 255:red, 74; green, 144; blue, 226 }  ,opacity=1 ] [align=left] {\begin{minipage}[lt]{8.67pt}\setlength\topsep{0pt}
$\displaystyle 2$
\end{minipage}};
\draw (270.29,57.85) node  [font=\small,color={rgb, 255:red, 74; green, 144; blue, 226 }  ,opacity=1 ] [align=left] {\begin{minipage}[lt]{8.67pt}\setlength\topsep{0pt}
$\displaystyle 2$
\end{minipage}};
\draw (55.29,57.85) node  [font=\small,color={rgb, 255:red, 74; green, 144; blue, 226 }  ,opacity=1 ] [align=left] {\begin{minipage}[lt]{8.67pt}\setlength\topsep{0pt}
$\displaystyle 2$
\end{minipage}};

\end{tikzpicture}

%% file: HS_standard.tex
\begin{tikzpicture}[x=0.75pt,y=0.75pt,yscale=-.75,xscale=.75]

\draw    (160,20) -- (160,60) ;
\draw [shift={(160,60)}, rotate = 90] [color={rgb, 255:red, 0; green, 0; blue, 0 }  ][fill={rgb, 255:red, 0; green, 0; blue, 0 }  ][line width=0.75]      (0, 0) circle [x radius= 2.01, y radius= 2.01]   ;
\draw [shift={(160,20)}, rotate = 90] [color={rgb, 255:red, 0; green, 0; blue, 0 }  ][fill={rgb, 255:red, 0; green, 0; blue, 0 }  ][line width=0.75]      (0, 0) circle [x radius= 2.01, y radius= 2.01]   ;
\draw    (160,20) -- (255,60) ;
\draw [shift={(255,60)}, rotate = 22.83] [color={rgb, 255:red, 0; green, 0; blue, 0 }  ][fill={rgb, 255:red, 0; green, 0; blue, 0 }  ][line width=0.75]      (0, 0) circle [x radius= 2.01, y radius= 2.01]   ;
\draw    (160,20) -- (65,60.06) ;
\draw [shift={(65,60.06)}, rotate = 157.14] [color={rgb, 255:red, 0; green, 0; blue, 0 }  ][fill={rgb, 255:red, 0; green, 0; blue, 0 }  ][line width=0.75]      (0, 0) circle [x radius= 2.01, y radius= 2.01]   ;
\draw    (160,60) -- (185,100) ;
\draw [shift={(185,100)}, rotate = 57.99] [color={rgb, 255:red, 0; green, 0; blue, 0 }  ][fill={rgb, 255:red, 0; green, 0; blue, 0 }  ][line width=0.75]      (0, 0) circle [x radius= 2.01, y radius= 2.01]   ;
\draw    (160,60) -- (135,100) ;
\draw [shift={(135,100)}, rotate = 122.01] [color={rgb, 255:red, 0; green, 0; blue, 0 }  ][fill={rgb, 255:red, 0; green, 0; blue, 0 }  ][line width=0.75]      (0, 0) circle [x radius= 2.01, y radius= 2.01]   ;
\draw    (255,60) -- (280,100) ;
\draw [shift={(280,100)}, rotate = 57.99] [color={rgb, 255:red, 0; green, 0; blue, 0 }  ][fill={rgb, 255:red, 0; green, 0; blue, 0 }  ][line width=0.75]      (0, 0) circle [x radius= 2.01, y radius= 2.01]   ;
\draw    (255,60) -- (230,100) ;
\draw [shift={(230,100)}, rotate = 122.01] [color={rgb, 255:red, 0; green, 0; blue, 0 }  ][fill={rgb, 255:red, 0; green, 0; blue, 0 }  ][line width=0.75]      (0, 0) circle [x radius= 2.01, y radius= 2.01]   ;
\draw    (230,100) -- (230,140) ;
\draw [shift={(230,140)}, rotate = 90] [color={rgb, 255:red, 0; green, 0; blue, 0 }  ][fill={rgb, 255:red, 0; green, 0; blue, 0 }  ][line width=0.75]      (0, 0) circle [x radius= 2.01, y radius= 2.01]   ;
\draw    (185,100) -- (200,140) ;
\draw [shift={(200,140)}, rotate = 69.44] [color={rgb, 255:red, 0; green, 0; blue, 0 }  ][fill={rgb, 255:red, 0; green, 0; blue, 0 }  ][line width=0.75]      (0, 0) circle [x radius= 2.01, y radius= 2.01]   ;
\draw    (185,100) -- (170,140) ;
\draw [shift={(170,140)}, rotate = 110.56] [color={rgb, 255:red, 0; green, 0; blue, 0 }  ][fill={rgb, 255:red, 0; green, 0; blue, 0 }  ][line width=0.75]      (0, 0) circle [x radius= 2.01, y radius= 2.01]   ;
\draw    (185,100) -- (185,140) ;
\draw [shift={(185,140)}, rotate = 90] [color={rgb, 255:red, 0; green, 0; blue, 0 }  ][fill={rgb, 255:red, 0; green, 0; blue, 0 }  ][line width=0.75]      (0, 0) circle [x radius= 2.01, y radius= 2.01]   ;
\draw    (135,100) -- (150,140) ;
\draw [shift={(150,140)}, rotate = 69.44] [color={rgb, 255:red, 0; green, 0; blue, 0 }  ][fill={rgb, 255:red, 0; green, 0; blue, 0 }  ][line width=0.75]      (0, 0) circle [x radius= 2.01, y radius= 2.01]   ;
\draw    (135,100) -- (120,140) ;
\draw [shift={(120,140)}, rotate = 110.56] [color={rgb, 255:red, 0; green, 0; blue, 0 }  ][fill={rgb, 255:red, 0; green, 0; blue, 0 }  ][line width=0.75]      (0, 0) circle [x radius= 2.01, y radius= 2.01]   ;
\draw    (135,100) -- (135,140) ;
\draw [shift={(135,140)}, rotate = 90] [color={rgb, 255:red, 0; green, 0; blue, 0 }  ][fill={rgb, 255:red, 0; green, 0; blue, 0 }  ][line width=0.75]      (0, 0) circle [x radius= 2.01, y radius= 2.01]   ;
\draw    (65,60.06) -- (90,100) ;
\draw [shift={(90,100)}, rotate = 57.96] [color={rgb, 255:red, 0; green, 0; blue, 0 }  ][fill={rgb, 255:red, 0; green, 0; blue, 0 }  ][line width=0.75]      (0, 0) circle [x radius= 2.01, y radius= 2.01]   ;
\draw    (65,60.06) -- (40,100) ;
\draw [shift={(40,100)}, rotate = 122.04] [color={rgb, 255:red, 0; green, 0; blue, 0 }  ][fill={rgb, 255:red, 0; green, 0; blue, 0 }  ][line width=0.75]      (0, 0) circle [x radius= 2.01, y radius= 2.01]   ;
\draw    (90,100) -- (90,140) ;
\draw [shift={(90,140)}, rotate = 90] [color={rgb, 255:red, 0; green, 0; blue, 0 }  ][fill={rgb, 255:red, 0; green, 0; blue, 0 }  ][line width=0.75]      (0, 0) circle [x radius= 2.01, y radius= 2.01]   ;
\draw    (40,100) -- (25,140) ;
\draw [shift={(25,140)}, rotate = 110.56] [color={rgb, 255:red, 0; green, 0; blue, 0 }  ][fill={rgb, 255:red, 0; green, 0; blue, 0 }  ][line width=0.75]      (0, 0) circle [x radius= 2.01, y radius= 2.01]   ;
\draw    (230,140) -- (245,180) ;
\draw [shift={(245,180)}, rotate = 69.44] [color={rgb, 255:red, 0; green, 0; blue, 0 }  ][fill={rgb, 255:red, 0; green, 0; blue, 0 }  ][line width=0.75]      (0, 0) circle [x radius= 2.01, y radius= 2.01]   ;
\draw    (230,140) -- (215,180) ;
\draw [shift={(215,180)}, rotate = 110.56] [color={rgb, 255:red, 0; green, 0; blue, 0 }  ][fill={rgb, 255:red, 0; green, 0; blue, 0 }  ][line width=0.75]      (0, 0) circle [x radius= 2.01, y radius= 2.01]   ;
\draw    (230,140) -- (230,180) ;
\draw [shift={(230,180)}, rotate = 90] [color={rgb, 255:red, 0; green, 0; blue, 0 }  ][fill={rgb, 255:red, 0; green, 0; blue, 0 }  ][line width=0.75]      (0, 0) circle [x radius= 2.01, y radius= 2.01]   ;
\draw    (90,140) -- (105,180) ;
\draw [shift={(105,180)}, rotate = 69.44] [color={rgb, 255:red, 0; green, 0; blue, 0 }  ][fill={rgb, 255:red, 0; green, 0; blue, 0 }  ][line width=0.75]      (0, 0) circle [x radius= 2.01, y radius= 2.01]   ;
\draw    (90,140) -- (75,180) ;
\draw [shift={(75,180)}, rotate = 110.56] [color={rgb, 255:red, 0; green, 0; blue, 0 }  ][fill={rgb, 255:red, 0; green, 0; blue, 0 }  ][line width=0.75]      (0, 0) circle [x radius= 2.01, y radius= 2.01]   ;
\draw    (90,140) -- (90,180) ;
\draw [shift={(90,180)}, rotate = 90] [color={rgb, 255:red, 0; green, 0; blue, 0 }  ][fill={rgb, 255:red, 0; green, 0; blue, 0 }  ][line width=0.75]      (0, 0) circle [x radius= 2.01, y radius= 2.01]   ;
\draw    (25,140) -- (40,180) ;
\draw [shift={(40,180)}, rotate = 69.44] [color={rgb, 255:red, 0; green, 0; blue, 0 }  ][fill={rgb, 255:red, 0; green, 0; blue, 0 }  ][line width=0.75]      (0, 0) circle [x radius= 2.01, y radius= 2.01]   ;
\draw    (25,140) -- (10,180) ;
\draw [shift={(10,180)}, rotate = 110.56] [color={rgb, 255:red, 0; green, 0; blue, 0 }  ][fill={rgb, 255:red, 0; green, 0; blue, 0 }  ][line width=0.75]      (0, 0) circle [x radius= 2.01, y radius= 2.01]   ;
\draw    (40,100) -- (55,140) ;
\draw [shift={(55,140)}, rotate = 69.44] [color={rgb, 255:red, 0; green, 0; blue, 0 }  ][fill={rgb, 255:red, 0; green, 0; blue, 0 }  ][line width=0.75]      (0, 0) circle [x radius= 2.01, y radius= 2.01]   ;

\draw (175.29,17.85) node  [font=\small,color={rgb, 255:red, 74; green, 144; blue, 226 }  ,opacity=1 ] [align=left] {\begin{minipage}[lt]{8.67pt}\setlength\topsep{0pt}
$\displaystyle 3$
\end{minipage}};
\draw (175.29,57.85) node  [font=\small,color={rgb, 255:red, 74; green, 144; blue, 226 }  ,opacity=1 ] [align=left] {\begin{minipage}[lt]{8.67pt}\setlength\topsep{0pt}
$\displaystyle 2$
\end{minipage}};
\draw (201.29,97.85) node  [font=\small,color={rgb, 255:red, 74; green, 144; blue, 226 }  ,opacity=1 ] [align=left] {\begin{minipage}[lt]{8.67pt}\setlength\topsep{0pt}
$\displaystyle 1$
\end{minipage}};
\draw (150.29,97.85) node  [font=\small,color={rgb, 255:red, 74; green, 144; blue, 226 }  ,opacity=1 ] [align=left] {\begin{minipage}[lt]{8.67pt}\setlength\topsep{0pt}
$\displaystyle 1$
\end{minipage}};
\draw (245.29,97.85) node  [font=\small,color={rgb, 255:red, 74; green, 144; blue, 226 }  ,opacity=1 ] [align=left] {\begin{minipage}[lt]{8.67pt}\setlength\topsep{0pt}
$\displaystyle 1$
\end{minipage}};
\draw (105.29,97.85) node  [font=\small,color={rgb, 255:red, 74; green, 144; blue, 226 }  ,opacity=1 ] [align=left] {\begin{minipage}[lt]{8.67pt}\setlength\topsep{0pt}
$\displaystyle 1$
\end{minipage}};
\draw (55.29,97.85) node  [font=\small,color={rgb, 255:red, 74; green, 144; blue, 226 }  ,opacity=1 ] [align=left] {\begin{minipage}[lt]{8.67pt}\setlength\topsep{0pt}
$\displaystyle 1$
\end{minipage}};
\draw (40.29,137.85) node  [font=\small,color={rgb, 255:red, 74; green, 144; blue, 226 }  ,opacity=1 ] [align=left] {\begin{minipage}[lt]{8.67pt}\setlength\topsep{0pt}
$\displaystyle 1$
\end{minipage}};
\draw (105.29,137.85) node  [font=\small,color={rgb, 255:red, 74; green, 144; blue, 226 }  ,opacity=1 ] [align=left] {\begin{minipage}[lt]{8.67pt}\setlength\topsep{0pt}
$\displaystyle 1$
\end{minipage}};
\draw (245.29,137.85) node  [font=\small,color={rgb, 255:red, 74; green, 144; blue, 226 }  ,opacity=1 ] [align=left] {\begin{minipage}[lt]{8.67pt}\setlength\topsep{0pt}
$\displaystyle 1$
\end{minipage}};
\draw (270.29,57.85) node  [font=\small,color={rgb, 255:red, 74; green, 144; blue, 226 }  ,opacity=1 ] [align=left] {\begin{minipage}[lt]{8.67pt}\setlength\topsep{0pt}
$\displaystyle 1$
\end{minipage}};
\draw (55.29,57.85) node  [font=\small,color={rgb, 255:red, 74; green, 144; blue, 226 }  ,opacity=1 ] [align=left] {\begin{minipage}[lt]{8.67pt}\setlength\topsep{0pt}
$\displaystyle 2$
\end{minipage}};

\end{tikzpicture}

%% file: HS_rigid.tex
\begin{tikzpicture}[x=0.75pt,y=0.75pt,yscale=-.75,xscale=.75]

\draw    (160,20) -- (160,60) ;
\draw [shift={(160,60)}, rotate = 90] [color={rgb, 255:red, 0; green, 0; blue, 0 }  ][fill={rgb, 255:red, 0; green, 0; blue, 0 }  ][line width=0.75]      (0, 0) circle [x radius= 2.01, y radius= 2.01]   ;
\draw [shift={(160,20)}, rotate = 90] [color={rgb, 255:red, 0; green, 0; blue, 0 }  ][fill={rgb, 255:red, 0; green, 0; blue, 0 }  ][line width=0.75]      (0, 0) circle [x radius= 2.01, y radius= 2.01]   ;
\draw    (160,20) -- (255,60) ;
\draw [shift={(255,60)}, rotate = 22.83] [color={rgb, 255:red, 0; green, 0; blue, 0 }  ][fill={rgb, 255:red, 0; green, 0; blue, 0 }  ][line width=0.75]      (0, 0) circle [x radius= 2.01, y radius= 2.01]   ;
\draw    (160,20) -- (65,60.06) ;
\draw [shift={(65,60.06)}, rotate = 157.14] [color={rgb, 255:red, 0; green, 0; blue, 0 }  ][fill={rgb, 255:red, 0; green, 0; blue, 0 }  ][line width=0.75]      (0, 0) circle [x radius= 2.01, y radius= 2.01]   ;
\draw    (160,60) -- (185,100) ;
\draw [shift={(185,100)}, rotate = 57.99] [color={rgb, 255:red, 0; green, 0; blue, 0 }  ][fill={rgb, 255:red, 0; green, 0; blue, 0 }  ][line width=0.75]      (0, 0) circle [x radius= 2.01, y radius= 2.01]   ;
\draw    (160,60) -- (135,100) ;
\draw [shift={(135,100)}, rotate = 122.01] [color={rgb, 255:red, 0; green, 0; blue, 0 }  ][fill={rgb, 255:red, 0; green, 0; blue, 0 }  ][line width=0.75]      (0, 0) circle [x radius= 2.01, y radius= 2.01]   ;
\draw    (255,60) -- (280,100) ;
\draw [shift={(280,100)}, rotate = 57.99] [color={rgb, 255:red, 0; green, 0; blue, 0 }  ][fill={rgb, 255:red, 0; green, 0; blue, 0 }  ][line width=0.75]      (0, 0) circle [x radius= 2.01, y radius= 2.01]   ;
\draw    (255,60) -- (230,100) ;
\draw [shift={(230,100)}, rotate = 122.01] [color={rgb, 255:red, 0; green, 0; blue, 0 }  ][fill={rgb, 255:red, 0; green, 0; blue, 0 }  ][line width=0.75]      (0, 0) circle [x radius= 2.01, y radius= 2.01]   ;
\draw    (230,100) -- (230,140) ;
\draw [shift={(230,140)}, rotate = 90] [color={rgb, 255:red, 0; green, 0; blue, 0 }  ][fill={rgb, 255:red, 0; green, 0; blue, 0 }  ][line width=0.75]      (0, 0) circle [x radius= 2.01, y radius= 2.01]   ;
\draw    (185,100) -- (200,140) ;
\draw [shift={(200,140)}, rotate = 69.44] [color={rgb, 255:red, 0; green, 0; blue, 0 }  ][fill={rgb, 255:red, 0; green, 0; blue, 0 }  ][line width=0.75]      (0, 0) circle [x radius= 2.01, y radius= 2.01]   ;
\draw    (185,100) -- (170,140) ;
\draw [shift={(170,140)}, rotate = 110.56] [color={rgb, 255:red, 0; green, 0; blue, 0 }  ][fill={rgb, 255:red, 0; green, 0; blue, 0 }  ][line width=0.75]      (0, 0) circle [x radius= 2.01, y radius= 2.01]   ;
\draw    (185,100) -- (185,140) ;
\draw [shift={(185,140)}, rotate = 90] [color={rgb, 255:red, 0; green, 0; blue, 0 }  ][fill={rgb, 255:red, 0; green, 0; blue, 0 }  ][line width=0.75]      (0, 0) circle [x radius= 2.01, y radius= 2.01]   ;
\draw    (135,100) -- (150,140) ;
\draw [shift={(150,140)}, rotate = 69.44] [color={rgb, 255:red, 0; green, 0; blue, 0 }  ][fill={rgb, 255:red, 0; green, 0; blue, 0 }  ][line width=0.75]      (0, 0) circle [x radius= 2.01, y radius= 2.01]   ;
\draw    (135,100) -- (120,140) ;
\draw [shift={(120,140)}, rotate = 110.56] [color={rgb, 255:red, 0; green, 0; blue, 0 }  ][fill={rgb, 255:red, 0; green, 0; blue, 0 }  ][line width=0.75]      (0, 0) circle [x radius= 2.01, y radius= 2.01]   ;
\draw    (135,100) -- (135,140) ;
\draw [shift={(135,140)}, rotate = 90] [color={rgb, 255:red, 0; green, 0; blue, 0 }  ][fill={rgb, 255:red, 0; green, 0; blue, 0 }  ][line width=0.75]      (0, 0) circle [x radius= 2.01, y radius= 2.01]   ;
\draw    (65,60.06) -- (90,100) ;
\draw [shift={(90,100)}, rotate = 57.96] [color={rgb, 255:red, 0; green, 0; blue, 0 }  ][fill={rgb, 255:red, 0; green, 0; blue, 0 }  ][line width=0.75]      (0, 0) circle [x radius= 2.01, y radius= 2.01]   ;
\draw    (65,60.06) -- (40,100) ;
\draw [shift={(40,100)}, rotate = 122.04] [color={rgb, 255:red, 0; green, 0; blue, 0 }  ][fill={rgb, 255:red, 0; green, 0; blue, 0 }  ][line width=0.75]      (0, 0) circle [x radius= 2.01, y radius= 2.01]   ;
\draw    (90,100) -- (90,140) ;
\draw [shift={(90,140)}, rotate = 90] [color={rgb, 255:red, 0; green, 0; blue, 0 }  ][fill={rgb, 255:red, 0; green, 0; blue, 0 }  ][line width=0.75]      (0, 0) circle [x radius= 2.01, y radius= 2.01]   ;
\draw    (40,100) -- (25,140) ;
\draw [shift={(25,140)}, rotate = 110.56] [color={rgb, 255:red, 0; green, 0; blue, 0 }  ][fill={rgb, 255:red, 0; green, 0; blue, 0 }  ][line width=0.75]      (0, 0) circle [x radius= 2.01, y radius= 2.01]   ;
\draw    (230,140) -- (245,180) ;
\draw [shift={(245,180)}, rotate = 69.44] [color={rgb, 255:red, 0; green, 0; blue, 0 }  ][fill={rgb, 255:red, 0; green, 0; blue, 0 }  ][line width=0.75]      (0, 0) circle [x radius= 2.01, y radius= 2.01]   ;
\draw    (230,140) -- (215,180) ;
\draw [shift={(215,180)}, rotate = 110.56] [color={rgb, 255:red, 0; green, 0; blue, 0 }  ][fill={rgb, 255:red, 0; green, 0; blue, 0 }  ][line width=0.75]      (0, 0) circle [x radius= 2.01, y radius= 2.01]   ;
\draw    (230,140) -- (230,180) ;
\draw [shift={(230,180)}, rotate = 90] [color={rgb, 255:red, 0; green, 0; blue, 0 }  ][fill={rgb, 255:red, 0; green, 0; blue, 0 }  ][line width=0.75]      (0, 0) circle [x radius= 2.01, y radius= 2.01]   ;
\draw    (90,140) -- (105,180) ;
\draw [shift={(105,180)}, rotate = 69.44] [color={rgb, 255:red, 0; green, 0; blue, 0 }  ][fill={rgb, 255:red, 0; green, 0; blue, 0 }  ][line width=0.75]      (0, 0) circle [x radius= 2.01, y radius= 2.01]   ;
\draw    (90,140) -- (75,180) ;
\draw [shift={(75,180)}, rotate = 110.56] [color={rgb, 255:red, 0; green, 0; blue, 0 }  ][fill={rgb, 255:red, 0; green, 0; blue, 0 }  ][line width=0.75]      (0, 0) circle [x radius= 2.01, y radius= 2.01]   ;
\draw    (90,140) -- (90,180) ;
\draw [shift={(90,180)}, rotate = 90] [color={rgb, 255:red, 0; green, 0; blue, 0 }  ][fill={rgb, 255:red, 0; green, 0; blue, 0 }  ][line width=0.75]      (0, 0) circle [x radius= 2.01, y radius= 2.01]   ;
\draw    (25,140) -- (40,180) ;
\draw [shift={(40,180)}, rotate = 69.44] [color={rgb, 255:red, 0; green, 0; blue, 0 }  ][fill={rgb, 255:red, 0; green, 0; blue, 0 }  ][line width=0.75]      (0, 0) circle [x radius= 2.01, y radius= 2.01]   ;
\draw    (25,140) -- (10,180) ;
\draw [shift={(10,180)}, rotate = 110.56] [color={rgb, 255:red, 0; green, 0; blue, 0 }  ][fill={rgb, 255:red, 0; green, 0; blue, 0 }  ][line width=0.75]      (0, 0) circle [x radius= 2.01, y radius= 2.01]   ;
\draw    (40,100) -- (55,140) ;
\draw [shift={(55,140)}, rotate = 69.44] [color={rgb, 255:red, 0; green, 0; blue, 0 }  ][fill={rgb, 255:red, 0; green, 0; blue, 0 }  ][line width=0.75]      (0, 0) circle [x radius= 2.01, y radius= 2.01]   ;

\draw (175.29,17.85) node  [font=\small,color={rgb, 255:red, 74; green, 144; blue, 226 }  ,opacity=1 ] [align=left] {\begin{minipage}[lt]{8.67pt}\setlength\topsep{0pt}
$\displaystyle 2$
\end{minipage}};
\draw (175.29,57.85) node  [font=\small,color={rgb, 255:red, 74; green, 144; blue, 226 }  ,opacity=1 ] [align=left] {\begin{minipage}[lt]{8.67pt}\setlength\topsep{0pt}
$\displaystyle 2$
\end{minipage}};
\draw (201.29,97.85) node  [font=\small,color={rgb, 255:red, 74; green, 144; blue, 226 }  ,opacity=1 ] [align=left] {\begin{minipage}[lt]{8.67pt}\setlength\topsep{0pt}
$\displaystyle 1$
\end{minipage}};
\draw (150.29,97.85) node  [font=\small,color={rgb, 255:red, 74; green, 144; blue, 226 }  ,opacity=1 ] [align=left] {\begin{minipage}[lt]{8.67pt}\setlength\topsep{0pt}
$\displaystyle 1$
\end{minipage}};
\draw (245.29,97.85) node  [font=\small,color={rgb, 255:red, 74; green, 144; blue, 226 }  ,opacity=1 ] [align=left] {\begin{minipage}[lt]{8.67pt}\setlength\topsep{0pt}
$\displaystyle 1$
\end{minipage}};
\draw (105.29,97.85) node  [font=\small,color={rgb, 255:red, 74; green, 144; blue, 226 }  ,opacity=1 ] [align=left] {\begin{minipage}[lt]{8.67pt}\setlength\topsep{0pt}
$\displaystyle 1$
\end{minipage}};
\draw (55.29,97.85) node  [font=\small,color={rgb, 255:red, 74; green, 144; blue, 226 }  ,opacity=1 ] [align=left] {\begin{minipage}[lt]{8.67pt}\setlength\topsep{0pt}
$\displaystyle 1$
\end{minipage}};
\draw (40.29,137.85) node  [font=\small,color={rgb, 255:red, 74; green, 144; blue, 226 }  ,opacity=1 ] [align=left] {\begin{minipage}[lt]{8.67pt}\setlength\topsep{0pt}
$\displaystyle 1$
\end{minipage}};
\draw (105.29,137.85) node  [font=\small,color={rgb, 255:red, 74; green, 144; blue, 226 }  ,opacity=1 ] [align=left] {\begin{minipage}[lt]{8.67pt}\setlength\topsep{0pt}
$\displaystyle 1$
\end{minipage}};
\draw (245.29,137.85) node  [font=\small,color={rgb, 255:red, 74; green, 144; blue, 226 }  ,opacity=1 ] [align=left] {\begin{minipage}[lt]{8.67pt}\setlength\topsep{0pt}
$\displaystyle 1$
\end{minipage}};
\draw (270.29,57.85) node  [font=\small,color={rgb, 255:red, 74; green, 144; blue, 226 }  ,opacity=1 ] [align=left] {\begin{minipage}[lt]{8.67pt}\setlength\topsep{0pt}
$\displaystyle 1$
\end{minipage}};
\draw (55.29,57.85) node  [font=\small,color={rgb, 255:red, 74; green, 144; blue, 226 }  ,opacity=1 ] [align=left] {\begin{minipage}[lt]{8.67pt}\setlength\topsep{0pt}
$\displaystyle 2$
\end{minipage}};

\end{tikzpicture}